\documentclass[11pt]{article}
\usepackage{amsmath,amsthm,amssymb,amsfonts,physics,hyperref}

\usepackage[english]{babel}
\usepackage[document]{ragged2e}
\usepackage[skip=10pt plus1pt, indent=40pt]{parskip}
\usepackage{longtable}
\usepackage[top=0.75in, bottom=0.5in, left=0.8in, right=0.8in,includefoot,heightrounded]{geometry}
\usepackage[utf8]{inputenc}
\usepackage[procnames]{listings}
\usepackage{pdflscape}
\usepackage{upgreek}
\usepackage{fancyhdr}
\usepackage{tikz-cd}
\usepackage[T1]{fontenc}
\usepackage[utf8]{inputenc}

\usepackage{mathtools}
\usepackage{dsfont}
\usepackage{bigints}
\usepackage{stmaryrd}
\usepackage{cases}
\usepackage{accents}

\mathtoolsset{showonlyrefs}

\numberwithin{equation}{section}
\numberwithin{figure}{section}

\newtheorem{maintheorem}{Theorem}

\newtheorem{defn0}{Definition}[section]
\newtheorem{prop0}[defn0]{Proposition}
\newtheorem{thm0}[defn0]{Theorem}
\newtheorem{lemma0}[defn0]{Lemma}
\newtheorem{corollary0}[defn0]{Corollary}
\newtheorem{example0}[defn0]{Example}
\newtheorem{remark0}[defn0]{Remark}
\newtheorem{conjecture0}[defn0]{Question}
\newtheorem{ansatz0}[defn0]{Ansatz}
\newtheorem{notation0}[defn0]{Notation}

\newenvironment{definition}{\bigskip \begin{defn0}}{\end{defn0}}
\newenvironment{proposition}{\bigskip \begin{prop0}}{\end{prop0}}
\newenvironment{theorem}{\bigskip \begin{thm0}}{\end{thm0}}
\newenvironment{lemma}{\bigskip \begin{lemma0}}{\end{lemma0}}

\newenvironment{remark}{ \bigskip \begin{remark0}}{\end{remark0}}
\newenvironment{conjecture}{\bigskip \begin{conjecture0}}{\end{conjecture0}}

\newenvironment{notation}{\bigskip \begin{notation0}}{\end{notation0}}

\newcommand{\R}{\mathds{R}}

\newcommand{\ad}{\operatorname{ad}}
\newcommand{\Ad}{\operatorname{Ad}}

\usepackage{overpic}
\usepackage{subcaption}
\usepackage{graphicx}

\usepackage{orcidlink}
\usepackage{hyperref}
\definecolor{Cerulean}{RGB}{29,172,214}
\hypersetup{
    colorlinks=true,
    linkcolor = black,
    citecolor = Cerulean,
    filecolor=magenta,      
    urlcolor=blue,
}

\title{Birkhoff normal forms, Dirac brackets and symplectic reduction}
\usepackage{authblk}

\makeatletter
\newcommand\blfootnote[1]{%
  \begingroup
  \renewcommand\thefootnote{}\footnote{#1}%
  \addtocounter{footnote}{-1}%
  \endgroup
}
\newcommand{\tpitchfork}{%
  \vbox{
    \baselineskip\z@skip
    \lineskip-.52ex
    \lineskiplimit\maxdimen
    \m@th
    \ialign{##\crcr\hidewidth\smash{$-$}\hidewidth\crcr$\pitchfork$\crcr}
  }%
}
\makeatother

\usepackage{biblatex}
\author{Jos\'e Lamas \orcidlink{0000-0002-1809-1823}}
\author{Lei Zhao \orcidlink{0000-0002-1754-8238}}
\affil{School of Mathematical Sciences, Dalian University of Technology, P.R. China}
\date{}
\begin{document}
\justifying
\maketitle

\begin{abstract}
\blfootnote{\textit{E-mail addresses}: \href{mailto:lamasrodriguezjose@dlut.edu.cn}{lamasrodriguezjose@dlut.edu.cn} (J.Lamas), \href{mailto:zhao1899@dlut.edu.cn}{zhao1899@dlut.edu.cn} (L.Zhao).}
Dirac brackets are widely used to study constrained Hamiltonian dynamics. In this paper we develop a Dirac-bracket approach to normal forms on momentum levels and relate it to symplectic reduction in the cases where reduction yields a (stratified) symplectic quotient. We consider a proper Hamiltonian $G$-action on a symplectic manifold $(M,\omega)$ with an equivariant momentum map $J$. We fix $\mu \in \mathfrak g^*$and work on $J^{-1}(\mu)$. For $G$-invariant Hamiltonians whose induced vector field on $J^{-1}(\mu)$ is tangent to a local $G_\mu$-slice, we show that the induced evolution on $J^{-1}(\mu)$ coincides with that defined by the Dirac bracket on a local second-class slice, and descends to the corresponding symplectic stratum of $J^{-1}(\mu)/G_\mu$.

As a main application we study Birkhoff normal forms near a relative equilibrium. When the quadratic part of a symmetric Hamiltonian is tangent to a local $G_\mu$-slice, a Birkhoff normal form can be constructed entirely on the manifold $J^{-1}(\mu)$, and it descends to a Birkhoff normal form for the reduced dynamics on the corresponding stratum, even when the reduced space is singular. We show that for a class of simple mechanical systems this condition holds automatically at a relative equilibrium. We illustrate the method on the double spherical pendulum.

Finally, we relate our results to Moser's constrained dynamics~\cite{1971430859838502419} by identifying Moser's constrained vector field with the Dirac Hamiltonian vector field. We show that, if the reduced Hamiltonian is near-integrable on a stratum, then its pullback to the symplectic slice is near-integrable with respect to the Dirac bracket, and vice versa. In particular, this provides a practical route to KAM-type results for the constrained dynamics.
\end{abstract}

\tableofcontents

\section{Introduction}\label{sec:introduction}

From Noether's theorem~\cite{Noether01011971} relating continuous symmetries and conserved quantities to modern geometric approaches, the use of symmetry as an organizing principle in dynamics has provided both conceptual insight and effective coordinates to understand the behavior of Hamiltonian systems. In Hamiltonian mechanics this viewpoint is nowadays known as \emph{theory of symmetry reduction}, where one exploits the presence of a symmetry group to lower the dimension of the phase space and identify the effective dynamics.

Concretely, given a Hamiltonian action of a Lie group $G$ on a symplectic manifold $(M,\omega)$, one associates to $G$ an equivariant momentum map $J\colon M\to \mathfrak g^*$, whose values encode the conserved quantities of the system (see for instance~\cite{guillemin1984normal, karshon1997centralizer, kostant2009orbits, marle1985modele}). The level sets of $J$ are invariant under any $G$-invariant Hamiltonian and carry a natural presymplectic form obtained by restricting $\omega$. Passing to the quotient of a regular momentum level $J^{-1}(\mu)$ by the corresponding isotropy group yields a reduced phase space $M_\mu:=J^{-1}(\mu)/G_\mu$ (assuming the $G_\mu$ action is free and proper) on which the dynamics is again Hamiltonian. This procedure, commonly known as \emph{regular symplectic reduction}, was put into its modern form by Marsden and Weinstein in~\cite{marsden1974reduction}, and further developed in many subsequent works (see for example~\cite{marsden2013introduction, ortega2013momentum}).

Beyond the regular case, momentum level sets and their quotients are typically only stratified spaces, with strata of different dimensions corresponding to different isotropy types. Each symplectic stratum carries a natural reduced symplectic form and a reduced Hamiltonian dynamics. This picture has been formalized in a series of works on \emph{singular symplectic reduction}: the reduced space is a stratified symplectic space \cite{sjarmaanlerman}. Each stratum is a symplectic manifold, the inclusions of strata are Poisson maps, and Hamiltonian dynamics of invariant functions restricts to Hamiltonian flows on each symplectic stratum (see for instance \cite{bates1997proper,Cushman_Śniatycki_2001,ORTEGA200651}).

From the viewpoint of local dynamical analysis, a recurring difficulty is that reduced spaces are often hard to work with explicitly. Even when one restricts to a smooth symplectic stratum, introducing canonical coordinates on the quotient is typically technical and problem-dependent: at singular values the quotient is not a manifold and there is no uniform coordinate description compatible with all strata. This issue becomes particularly important in normal form problems near equilibria or relative equilibria, where the standard normalization procedures are naturally formulated in canonical coordinates. In symmetric Hamiltonian systems, one frequently reduces first and then performs the normal form construction on the reduced space or on a symplectic stratum. 

The aim of this paper is to develop a complementary viewpoint in which local computations can be carried out on a smooth space ``upstairs'', while keeping a direct relation with the reduced dynamics (with a view of constraints in mind). A natural setting for such a viewpoint is provided by the momentum level itself. The restriction of $\omega$ to $J^{-1}(\mu)$ is generally presymplectic, with characteristic distribution generated by the $G_\mu$-orbits. As a consequence, the Hamiltonian vector field associated to an invariant Hamiltonian, when restricted to the momentum level, is naturally defined only modulo the characteristic directions. A classical way to obtain an intrinsic bracket description on a constrained manifold is given by Dirac's theory of second-class constraints \cite{dirac1958generalized}: locally, one models the constraint as the common zero set of constraint functions whose Poisson bracket matrix is invertible, and the associated Dirac bracket reproduces the intrinsic Poisson bracket on the constraint while remaining defined in the ambient space (see also the geometric interpretations in~\cite{Bursztyn_2013, dufour2005poisson}).

In the context of momentum levels, this leads to three closely related descriptions of the same effective dynamics on $J^{-1}(\mu)$:
\begin{itemize}
\item view the dynamics intrinsically on the momentum level $J^{-1}(\mu)$ using the restricted presymplectic form;
\item project to the reduced dynamics on a symplectic stratum of the quotient $J^{-1}(\mu)/G_\mu$;
\item pass to a local symplectic slice complementary to the characteristic distribution and describe the constrained dynamics by a Dirac bracket associated with a second-class constraint model.
\end{itemize}
This paper aims to clarify how these descriptions compare locally, and how this comparison may be exploited to study local normal forms for the constrained dynamics on the momentum level (via Dirac brackets) near relative equilibria, without introducing canonical coordinates on the quotient.

\paragraph{Organization of the paper.}
This paper is organized as follows. In Section~\ref{sec: setting and notation} we fix the geometric setting and we recall the constrained/Dirac framework used throughout the paper. In Section~\ref{sec: main results} we state the main results regarding the comparison between the dynamics on a momentum level, the Dirac-bracket dynamics on a local slice, and the reduced dynamics on a symplectic stratum. We also discuss the limiting singular regimes. In Section~\ref{sec: Proof of Thms} and Section~\ref{sec: Proof of mechsys} we prove the results from Section~\ref{sec: main results}. In Section~\ref{sec: double spherical pendulum} we illustrate these results on the double spherical pendulum. In Section~\ref{sec:Moser-extension} we treat Moser's viewpoint on constrained systems~\cite{1971430859838502419} and explain how the mechanisms developed in this paper extends his study on near-integrable systems.

\section{Setting and notation}\label{sec: setting and notation}
In this section we fix the notation and we recall the background needed to state and prove the main results in Section~\ref{sec: main results}. In Section~\ref{subsec:intro-dirac} we recall Dirac's treatment of second-class constraints and the associated Dirac bracket~\cite{dirac1958generalized}. In Section~\ref{subsec:intro-normalforms} we record the normal form viewpoint, in particular the issue of carrying out normalization procedures in the presence of constraints. Finally, in Section~\ref{subsec:intro-momentum} we summarize the basic constructions of Hamiltonian group actions, momentum maps and (possibly singular) symplectic reduction that will be used to compare dynamics on a momentum level, on a local Dirac slice, and on a reduced symplectic stratum.

\subsection{Constraint systems and Dirac brackets}\label{subsec:intro-dirac}

Many systems are naturally described on a subset of phase space determined by constraints. Typical examples include holonomic constraints in mechanics, conserved quantities imposed as conditions, and constraint manifolds that appear as momentum levels in symmetric Hamiltonian systems. In all these situations one would like to describe the dynamics intrinsically on the constraint manifold, and in particular to have a bracket formulation that avoids introducing explicit coordinates on the constrained space.


Let $(M,\omega)$ be a symplectic manifold. For each $f\in C^\infty(M)$, denote by $X_f$ its Hamiltonian vector field defined by $\iota_{X_f}\omega=-df$ and by $\{\cdot,\cdot\}_M$ the associated Poisson bracket. Let $\phi_1,\dots,\phi_k\in C^\infty(M)$ be independent constraint functions and set
\begin{equation}\label{eqn: N definition}
  N:=\{x\in M:\ \phi_1(x)=\cdots=\phi_k(x)=0\}.
\end{equation}
Following Dirac \cite{dirac1958generalized}, constraints are classified according to their Poisson brackets in the following way (see Figure~\ref{fig:first_second_class_constraints}).

\begin{definition}\label{def:first-second-class}
The constraints $\phi_1,\dots,\phi_k$ are called \emph{first-class} if
\[
  \{\phi_i,\phi_j\}_M\big|_{N}=0\qquad \text{for all }i,j,
\]
that is, their Poisson brackets vanish on the constraint manifold. They are called \emph{second-class} if the
matrix
\begin{equation}\label{eqn: matrix C}
  C(x):=\left(\{\phi_i,\phi_j\}_M(x)\right)_{i,j=1}^k = (C_{ij})_{i,j=1}^k
\end{equation}
is invertible for all $x\in N$.
\end{definition}

In this paper we work with the second-class case. Each tangent space $T_xN$ admits a canonical complement generated by the constraint vector fields $X_{\phi_i}(x)$'s, and one can project Hamiltonian vector fields to obtain an intrinsic bracket description of the constrained dynamics.
\begin{figure}[ht!]
  \centering
  \begin{subfigure}[t]{0.54\linewidth}
    \centering
    \includegraphics[width=\linewidth]{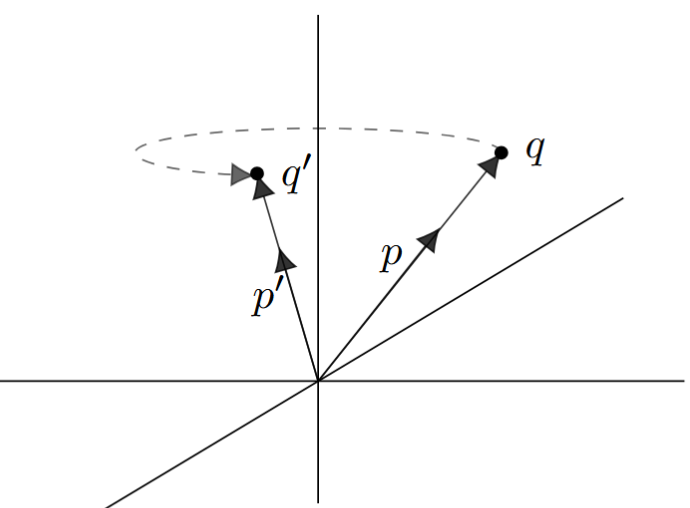}
  \end{subfigure}\hfill
  \begin{subfigure}[t]{0.44\linewidth}
    \centering
    \includegraphics[width=\linewidth]{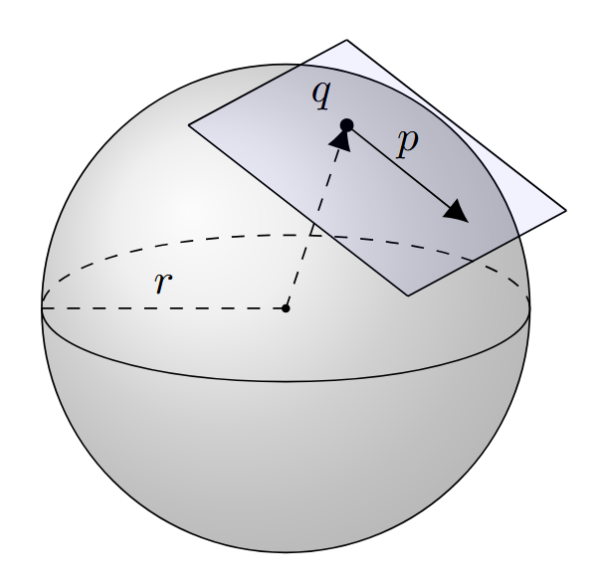}
  \end{subfigure}
  \caption{Illustration of Dirac's distinction between first-class and second-class constraints. On the left, the first-class constraint arises from the cotangent-lifted $SO(3)$-action on $T^*\mathds{R}^3$ and is given by the zero level of the momentum map, $J(q,p)=q\times p=0$. Its Hamiltonian vector fields are tangent to the $SO(3)$-orbits, so the constrained manifold retains the symmetry directions and is naturally presymplectic. On the right, the second-class constraint describes a particle constrained to the sphere $S_r^2$, with $\phi_1(q,p)=|q|^2-r^2=0$ fixing the position on the sphere and $\phi_2(q,p)=q\cdot p=0$ removing the normal momentum component. In contrast with the first-class case, these constraints cut out a genuine symplectic constrained phase space, whose dynamics is described by the Dirac bracket.}
  \label{fig:first_second_class_constraints}
\end{figure}
\begin{lemma}[\cite{dirac1958generalized}, see also \cite{marsden2013introduction}]\label{lemma:Dirac bracket}
Assume $\phi_1,\dots,\phi_k$ are second-class and write $C^{ij}$ for the entries of $C^{-1}$ in~\eqref{eqn: matrix C}. For $f\in C^\infty(M)$, define the \emph{Dirac projection} of $X_f$ along $N$ in~\eqref{eqn: N definition} by
\begin{equation}\label{eq:dirac-projection}
  \mathrm P_N(X_f)
  := X_f - \sum_{i,j=1}^k \{f,\phi_i\}_M\,C^{ij}\,X_{\phi_j}\qquad \text{on }N.
\end{equation}
Equivalently, $\mathrm P_N(X_f)$ is the unique vector field along $N$ that is tangent to $N$ and differs
from $X_f$ by a linear combination of the $X_{\phi_j}$'s.

The \emph{Dirac bracket} is the bracket on functions defined (near $N$) by
\[
  \{f,g\}_D := \big(\mathrm P_N(X_f)\big)[g]\Big|_{N}.
\]
Equivalently,
\begin{equation}\label{eq:dirac-bracket}
  \{f,g\}_D
  = \{f,g\}_M-\sum_{i,j=1}^k \{f,\phi_i\}_M\,C^{ij}\,\{\phi_j,g\}_M .
\end{equation}
\end{lemma}
By construction, $\mathrm P_N(X_f)$ is tangent to $N$ and thus preserves the constraints. Moreover, the restriction of $\{\cdot,\cdot\}_D$ to $N$ coincides with the Poisson bracket associated to $\omega|_N$. In particular, for a Hamiltonian $H$ the constrained dynamics on $N$ generated by $H|_N$ is described by the projected vector field $\mathrm P_N(X_H)$.

Modern geometric interpretations relate this construction to Dirac structures and Poisson geometry (see, for example, \cite{Bursztyn_2013,dufour2005poisson}). Dirac-type reduced brackets also appear beyond finite-dimensional mechanics \cite{chandre2012hamiltonian,salmon1988semigeostrophic}.

\begin{remark}\label{rem:first-class}
If the constraints $\phi_1,\dots,\phi_k$ are first-class then the constraint matrix $C(x)$ in~\eqref{eqn: matrix C} is singular, and one cannot define a Dirac bracket from $\{\phi_i\}$ alone. If, moreover, $N=\{\phi_1=\cdots=\phi_k=0\}$ is a regular submanifold, then $N$ is coisotropic and the pullback $\omega_N:=\iota^*\omega$ is presymplectic, with characteristic foliation $\mathcal F=\ker(\omega_N)$. The corresponding intrinsic reduction may then be described as follows:
\begin{itemize}
  \item If the reduced space $\bar N:=N/\mathcal F$ is smooth, it carries a natural Poisson structure for which the projection $\pi:N\to\bar N$ is Poisson. Hence the reduced dynamics is Hamiltonian on $(\bar N,\{\cdot,\cdot\}_{\bar N})$.
  \item If $\bar N$ is not smooth, one may still describe the reduced Poisson algebra as the Poisson algebra of $\mathcal F$-basic functions on $N$ (i.e. functions $f\in C^\infty(N)$ satisfying $df|_{\ker(\omega_N)}=0$). In the momentum map setting this recovers the induced Poisson structure on the (generally singular) reduced space, which is a stratified symplectic space (see~\cite{ArmsGotayJennings1990,bates1997proper,sjarmaanlerman}).
\end{itemize}
In the presence of symmetry, this framework recovers regular and singular symplectic reduction by a momentum map (see, for instance,~\cite{marsden1974reduction,ortega2013momentum}). This includes, in particular, Hamiltonian actions of abelian groups, and more generally the case of coadjoint-fixed values of~$\mu$. For a general Hamiltonian $G$-action, we may consider $J^{-1}(\mu)$ as a coisotropic submanifold whose characteristic foliation is generated by the infinitesimal $\mathfrak g_\mu$-action. The reduced space is then the quotient $J^{-1}(\mu)/G_\mu$, in either the regular or singular case (see Section~\ref{subsec:intro-momentum} below).

\end{remark}

\subsection{Normal forms in constrained dynamics}\label{subsec:intro-normalforms}
Normal form theory provides a tool to locally simplify Hamiltonian systems near an equilibrium (or a relative equilibrium). Close to such a point one attempts to conjugate the Hamiltonian by successive canonical transformations, order by order, so that only the resonant terms with respect to the linearized dynamics remain. The resulting \emph{Birkhoff normal form} is a powerful invariant which captures rich dynamical features, such as linear stability, resonances, and persistence of quasi-periodic motions.

In the presence of constraints, however, the standard procedure does not automatically respect the constraint dynamics. Thus, a normal form constructed on the ambient space may fail to preserve a given constraint manifold, or may not project correctly to a reduced phase space. This leads naturally to the problem of constructing normal forms that are adapted to a constraint submanifold, or to a symplectic stratum of a singular reduced space, rather than only to the ambient symplectic manifold. Various constrained or reduced normal form schemes have been developed in this spirit (see for instance~\cite{junginger2017construction, van2017geometry, van1986constrained,meyer2018normalization}), but a formulation that directly exploits the Dirac bracket on a momentum level appears to be still missing.

Several approaches address related issues by exploiting invariants or constraints directly in the normalization scheme. We mention a few representative examples that motivate the questions studied here.

\paragraph{Constrained normalization in Kepler-type dynamics.}
In~\cite{van1986constrained}, the authors developed a normalization procedure for the perturbed Kepler problem that is built to respect a constraint structure arising from the geometric relations of the Kepler flow. The point is not only to simplify the Hamiltonian, but to do so while keeping the relevant constrained manifold invariant under the transformations used in the normalization. A related approach based on invariants is developed in~\cite{meyer2018normalization}, where the normal form is organized using explicit polynomial invariants of the Kepler dynamics in higher dimension. These works show that, even in problems where canonical coordinates exist globally, the normal form construction can be forced to respect additional geometric structure.

\paragraph{Normal forms and reduction.}
For symmetric Hamiltonian systems, normal forms are often computed after reduction. In \cite{CIFTCI20141} the authors analyze in detail how cotangent bundle reduction interacts with Poincar\'e-Birkhoff normal forms, comparing the normal form constructed before reduction with the one constructed on the reduced space. This comparison is close in spirit to the present work: the reduced space carries the relevant dynamics, but computations are simpler on a smooth space before doing quotient. Our results give a general criterion under which this comparison can be made directly on a momentum level via the Dirac bracket.


These examples share a recurring difficulty: the space on which the effective dynamics lives is often not presented in canonical coordinates, while normal form procedures are most naturally formulated in such coordinates. The main results of this paper address this difficulty in the setting of symmetric Hamiltonian systems and symplectic reduction.


\subsection{Symmetric Hamiltonian systems, momentum maps and symplectic reduction}\label{subsec:intro-momentum}
Let $G$ be a Lie group acting properly on a symplectic manifold $(M,\omega)$ by symplectomorphisms. If the action is Hamiltonian, it admits an equivariant momentum map $J:M\to\mathfrak g^*$ characterized by $d\langle J,\xi\rangle=\iota_{\xi_M}\omega$ for $\xi\in\mathfrak g$. For $G$-invariant Hamiltonians, $J$ is conserved along the flow (Noether's theorem, see~\cite{Noether01011971}). Beyond conservation laws, momentum maps also control the local geometry of the action through local normal forms and slice models. In particular, near a given orbit one has canonical local descriptions of the symplectic form and the momentum map (see for example \cite{guillemin1984normal,karshon1997centralizer,kostant2009orbits,marle1985modele}).

To describe the geometry of the symmetry action near a point, we will use the local model provided by the slice theorem for proper actions~\cite{palaisSlice} (see also~\cite{ortega2013momentum}), which we recall here.
\begin{definition}\label{def: local slice}
Let a Lie group $K$ act properly on a manifold $N$ and let $x_0\in N$ with stabilizer $K_{x_0}:=\{k\in K: k\cdot x_0=x_0\}$. A (local) \emph{slice} at $x_0$ (see Figure~\ref{fig:slice}) is an embedded submanifold $S\subset N$ through $x_0$ such that 
\begin{itemize}
\item $S$ is $K_{x_0}$-invariant;
\item after shrinking $S$ if necessary, the saturation $K\cdot S$ is an open neighborhood of $x_0$ in $N$ and the natural map
\[
K\times_{K_{x_0}} S \;\longrightarrow\; K\cdot S, \qquad [k,s]\mapsto k\cdot s,
\]
is a $K$-equivariant diffeomorphism.
\end{itemize}
In particular, after possibly shrinking $S$, for each $x\in S$ one has a direct sum decomposition
\begin{equation}\label{eqn: slice decomposition}
T_x N  = T_x(K\cdot x) \oplus T_x S,
\end{equation}
so that $T_xS$ provides a choice of directions complementary to the orbit directions.
\end{definition}

\begin{figure}[ht]
    \centering
    \includegraphics[width=0.7\linewidth]{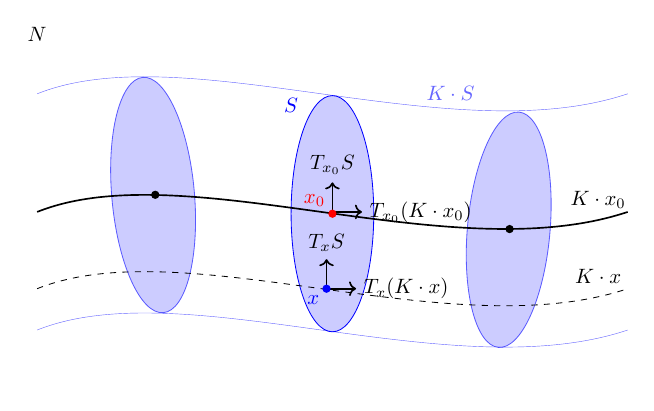}
    \caption{Local slice for a proper $K$-action. The slice $S$ passes through $x_0$ and is transverse to the orbit $K\cdot x_0$. Its saturation $K\cdot S$ is an open neighborhood of $x_0$, locally modeled by the associated bundle $K\times_{K_{x_0}} S$. In particular, $T_xN = T_x(K\cdot x)\oplus T_xS$ for $x\in S$ near $x_0$.}
    \label{fig:slice}
\end{figure}

Because $H$ is $G$-invariant, the Hamiltonian vector field $X_H$ is $G$-equivariant, and the restriction of $X_H$ to an orbit $G\cdot x$ is determined by infinitesimal generators $\xi_M$. We will be concerned with a point $x_0$ for which $X_H(x_0)$ is tangent to the group orbit, so that the corresponding trajectory remains in $G\cdot x_0$. This is the standard notion of a relative equilibrium, which we recall here.
\begin{definition}\label{def:rel-eq}
Let $H \in C^\infty(M)^G$ be a $G$-invariant Hamiltonian. A point $x_0 \in M$ is called a
\emph{relative equilibrium} of $H$ if its Hamiltonian trajectory is contained in its $G$-orbit, that is, if
there exists $\xi \in \mathfrak g$ such that
\[
  X_H(x_0) = \xi_M(x_0),
\]
where $\xi_M$ is the infinitesimal generator of the $G$-action associated to $\xi$. 
\end{definition}
Equivalently, the integral curve of $X_H$ through $x_0$ coincides with the group orbit $G\cdot x_0$ up to reparametrization. Abusing notation, we will refer both to the point $x_0$ and to the orbit $G\cdot x_0$ as a relative equilibrium when no confusion can arise.

Local dynamics near a relative equilibrium is described by a quadratic Hamiltonian on a symplectic slice together with higher-order terms. This is the setting in which we apply Birkhoff normal form methods in the presence of reduction: the comparison between the normal form on a slice, the normal form on a momentum level, and the normal form on the reduced stratum depends on how the quadratic part relates the slice directions with the symmetry directions.

From the viewpoint of normal forms, the main issue is practical: even when one works inside a smooth symplectic stratum, introducing explicit canonical coordinates on the quotient is often technical and depends strongly on the problem. This motivates working on the smooth momentum level itself and using a second-class constraint description (see Definition~\ref{def:first-second-class}) to encode the reduced Poisson structure. 

\section{Main results}\label{sec: main results}
Let $(M,\omega)$ be a $2m$-dimensional symplectic manifold endowed with a proper Hamiltonian $G$-action with equivariant momentum map $J:M\to\mathfrak g^*$. Fix $\mu\in\mathfrak g^*$ and set
\begin{equation}\label{eqn: notation J, Gmu, Mmu}
  J_\mu := J^{-1}(\mu), \qquad G_\mu := \{g\in G:\ \Ad_g^*\mu=\mu\}, \qquad M_\mu := J_\mu/G_\mu,
\end{equation}
with quotient map
\begin{equation}\label{eqn: quotient map}
  \pi_{\mu}:J_\mu\to M_\mu.
\end{equation}

\begin{remark}\label{remark:coisotropic}
At every point $x\in J_\mu$ where $J_\mu$ is a smooth submanifold of $M$, the subspace $T_xJ_\mu\subset (T_xM,\omega_x)$ is coisotropic and its characteristic distribution $\ker(\iota^*\omega)_x$ is generated by the infinitesimal $\mathfrak g_\mu$-action at $x$ (see, e.g.,~\cite{marsden2013introduction,10.1093/oso/9780198794899.001.0001}).
\end{remark}
Since $\iota^*\omega$ is in general only presymplectic on $J_\mu$ (where $\iota:J_\mu\hookrightarrow M$), there is no canonical Poisson bracket on all of $C^\infty(J_\mu)$. We therefore restrict to the subalgebra of \emph{admissible functions}
\begin{equation}\label{eqn: admissible functions}
C^\infty_{\mathrm{adm}}(J_\mu)
:=\left\{ f\in C^\infty(J_\mu)\ :\ df(v)=0\ \text{for all }v\in \ker(\iota^*\omega)\right\}.
\end{equation}
Equivalently, $f$ is admissible if and only if there exists a vector field $X_f$ on $J_\mu$ such that $\iota_{X_f}(\iota^*\omega)=-df$. In this case $X_f$ is unique modulo $\ker(\iota^*\omega)$.

For later use, we record a convenient standard extrinsic characterization (see for instance~\cite{ArmsGotayJennings1990,LaurentGengouxPichereauVanhaeckePoissonStructures}): $f\in C^\infty(J_\mu)$ is admissible if and only if it admits an extension $F\in C^\infty(M)$ with $f=F|_{J_\mu}$ 
and $X_F|_{J_\mu}$ is tangent to $J_\mu$. Given $f,g\in C^\infty_{\mathrm{adm}}(J_\mu)$, choose admissible extensions $F,G\in C^\infty(M)$ and define
\begin{equation}\label{eqn:induced-bracket-J}
  \{f,g\}_{J_\mu}:=\{F,G\}_M\big|_{J_\mu}.
\end{equation}
This is independent of the chosen admissible extensions and equips $C^\infty_{\mathrm{adm}}(J_\mu)$ with a Poisson bracket.

\begin{notation}\label{not:admissible-convention}
Throughout the paper, the bracket $\{\cdot,\cdot\}_{J_\mu}$ is only applied to admissible functions, and
$X_f:=\{\cdot,f\}_{J_\mu}$ denotes the associated Hamiltonian derivation on $C^\infty_{\mathrm{adm}}(J_\mu)$.
When a function on $J_\mu$ is written inside $\{\cdot,\cdot\}_{J_\mu}$, it is implicitly assumed to be admissible.
\end{notation}

We recall that $M_\mu$ is in general a stratified symplectic space. In particular, each symplectic stratum $\mathcal S\subset M_\mu$ carries a reduced symplectic form and hence a Poisson bracket, which we denote by $\{\cdot,\cdot\}_{\mathcal S}$.

Given a $G$-invariant Hamiltonian function $F\in C^\infty(M)^G$, we denote by $F|_{J_\mu}\in C_{\mathrm{adm}}^\infty(J_\mu)$ its restriction to $J_\mu$, and by $\tilde F\in C^\infty(M_\mu)$ the induced function on any symplectic stratum $\mathcal S\subset M_\mu$. We define the linear operators
\begin{equation}\label{eqn: linoperators}
\mathcal L_{J_\mu}^F := \{F|_{J_\mu},\cdot\}_{J_\mu},\qquad
\mathcal L_D^F := \{F|_{J_\mu},\cdot\}_D,\qquad
\mathcal L_S^F := \{\tilde F,\cdot\}_{\mathcal S}.
\end{equation}
Here, the brackets $\{\cdot,\cdot\}_D$ and $\{\cdot,\cdot\}_{J_\mu}$ are defined in~\eqref{eq:dirac-bracket} and~\eqref{eqn:induced-bracket-J} respectively.

\begin{definition}\label{def:drift-free}
We say that $F\in C^\infty(M)^G$ is \emph{drift-free near $x_0 \in J_\mu$} if there exists a $G_\mu$-slice $C\subset J_\mu$ (see Definition~\ref{def: local slice}) through $x_0$ such that, after shrinking $C$ if necessary,
\[
X_F(x)\in T_xC \qquad \text{for all $x\in C$ near $x_0$}.
\]
Equivalently, in local Darboux coordinates $(\theta,I,u)$ adapted to the $G_\mu$-action and the momentum level set $J_\mu=\{I=0\}$, the drift-free condition is 
\begin{equation}\label{eqn: drift-free derivative condition} 
\partial_I F(\theta,0,u)=0 \quad\text{for all $(\theta,u)$ near $x_0$.} 
\end{equation}
\end{definition}

\begin{remark}
The terminology \emph{drift} comes from Krupa~\cite{krupa1990} (see also~\cite{KrupaMelbourne1995Asymptotic}),
who describes the \emph{drift along the group orbits} of an invariant vector field as its $G_\mu$-orbit-tangent
part in a slice decomposition. More precisely, let $X$ be a $G_\mu$-invariant vector field on $J_\mu$ and let $C\subset J_\mu$ be a $G_\mu$-slice through $x_0$. Along $C$ one may write
\[
X|_C = X_{\mathrm{orb}} + X_C,
\]
where $X_C$ is tangent to $C$ and $X_{\mathrm{orb}}$ is tangent to the $G_\mu$-orbits. With this terminology, $X_{\mathrm{orb}}$ is referred to as the \emph{drift} of $X$ along the group orbits.
\end{remark}
We state the main results of this paper.
\bigskip
\begin{maintheorem}\label{thm:general_intertwining_JS}
Let $F\in C^\infty(M)^G$ be a $G$-invariant function. For $x_0 \in J_\mu$, let $\mathcal S\subset M_\mu$ be a symplectic stratum containing a point $[x_0]\in M_\mu$. Assume that $F$ is drift-free near $x_0$. Then there exist neighborhoods $U\subset J_\mu$ of $x_0$  and  $V\subset \mathcal S$ of $[x_0]$ such that 
\begin{equation}\label{eqn: LD =LJ}
\mathcal L_D^F =\mathcal L_{J_\mu}^F\quad \text{on } C^\infty(U),
\end{equation}
and therefore
\begin{equation}\label{eqn: main-intertwining-F} 
\pi_{\mu}^* \circ \mathcal L_S^F \;=\; \mathcal L_{D}^F \circ \pi_{\mu}^*\quad\text{on } C^\infty(V),
\end{equation}
where the operators $\mathcal L_{J_\mu}^F,\mathcal L_D^F$ and $\mathcal L_S^F$ are defined in~\eqref{eqn: linoperators} and $\pi_{\mu}$  in~\eqref{eqn: quotient map}. 
\end{maintheorem}
A direct consequence of this result is the following theorem, which shows that the Birkhoff normal form for the reduced dynamics can be defined as the one induced from a Birkhoff normal form on $J_\mu$.

\bigskip
\begin{maintheorem}\label{thm:main-BNF}
Let $H\in C^\infty(M)^G$ be a $G$-invariant Hamiltonian and let $x_0\in J_\mu = J^{-1}(\mu)$ be a relative equilibrium of $H$ (see Definition~\ref{def:rel-eq}) with momentum $\mu \in \mathfrak g^*$. Let $H_2$ be the quadratic part of $H$ at $x_0$, and let $\mathcal S\subset M_\mu$ be a symplectic stratum containing $[x_0]$. Assume that $H_2$ is drift-free near $x_0$. Denote by
\begin{equation}\label{eqn: LJ and LS H2}
\mathcal L_{J_\mu}:= \{H_2|_{J_\mu},\cdot\}_{J_\mu},\quad \mathcal L_S:=\{\tilde H_2,\cdot\}_{\mathcal S}
\end{equation}
the corresponding quadratic homological operators on $J_\mu$ and on $S$, where $\tilde H_2$ is the reduced quadratic Hamiltonian on $\mathcal S$. Then, in a neighborhood of $x_0$ in $J_\mu$, there exists a (formal) Birkhoff normal form $H_{\mathrm{BNF}}$ for $H|_{J_\mu}$ with respect to $\mathcal L_{J_\mu}$ such that the induced Hamiltonian $\tilde H_{\mathrm{BNF}}$ on $\mathcal S$ is a (formal) Birkhoff normal form for the reduced dynamics with respect to $\mathcal L_S$. 
\end{maintheorem}
\begin{remark}
Since the action is proper, $G_{x_0}$ is compact. Then, $\mathcal L_{J_\mu}$ commutes with averaging over $G_{x_0}$, so one may further consider the Birkhoff normal form $H_{\mathrm{BNF}}$ to be $G_{x_0}$-invariant. 

If in addition $G_\mu$ is compact, then one can construct a $G_\mu$-invariant Birkhoff normal form on $J_\mu$ by choosing the normalizing generating functions to be $G_\mu$-invariant at each order (e.g.\ by averaging them over $G_\mu$). The resulting normal form Hamiltonian $H_{\mathrm{BNF}}$ is $G_\mu$-invariant and therefore descends to a well-defined Hamiltonian on $M_\mu$, whose restriction to $\mathcal S$ is a Birkhoff normal form for the reduced dynamics.

\end{remark}
A natural question regarding Theorem~\ref{thm:main-BNF} is how restrictive its assumption (i.e. the drift-free condition on $H_2$) really is. The following theorem shows that this condition is satisfied for a large class of simple mechanical systems with symmetry. To state the result we recall the notion of simple mechanical system and of stationary locked inertia tensor from~\cite[Sec.~6.6]{ortega2013momentum}.
\begin{definition}\label{def:simple-mech-stationary}
A \emph{simple mechanical system with symmetry} is a quadruple $(Q,g,V,G)$ where
\begin{itemize}
    \item $Q$ is a smooth manifold (the configuration space);
    \item $g$ is a Riemannian metric on $Q$;
    \item $V\colon Q \to \mathds R$ is a smooth potential;
    \item $G$ is a Lie group acting properly on $Q$ by isometries of $g$, with $V$ being $G$-invariant.
\end{itemize}
The \emph{locked inertia tensor} associated to this system~\cite[Sec.~6.6.7]{ortega2013momentum} is the map
\begin{equation}\label{eqn: locked inertia tensor}
  \mathds I(q):\mathfrak g\to\mathfrak g^*,\qquad
  \langle \mathds I(q)\xi,\eta\rangle := g_q(\xi_Q(q),\eta_Q(q)),
\end{equation}
for $\xi,\eta\in\mathfrak g$, where $\xi_Q$ denotes the infinitesimal generator of the $G$-action. For fixed $\xi,\eta\in\mathfrak g$, define the associated scalar function
\begin{equation}\label{eqn: scalar function}
f_{\xi,\eta}\colon Q \to \mathds R,\qquad f_{\xi,\eta}(q) = \langle\mathds I(q)\,\xi,\eta\rangle.
\end{equation}
Let $K\subseteq G$ be a compact subgroup with Lie algebra $\mathfrak k$, and let $q_0\in Q$. We say that the locked inertia tensor is \emph{stationary along $K$ at $q_0$} if there exists a $K_{q_0}$-invariant submanifold $S \subset Q$ (a ``slice'' for the $K$-action) through $q_0$ such that

\begin{itemize}
  \item[\emph{(i)}] $T_{q_0}Q = T_{q_0}(K\cdot q_0) \oplus T_{q_0}S$, so $S$ is transverse to the $K$-orbit at $q_0$;
  \item[\emph{(ii)}] the saturation $K\cdot S$ is an invariant neighborhood of $K\cdot q_0$ in $Q$;
  \item[\emph{(iii)}] for every $\xi,\eta \in \mathfrak k$ and every smooth curve $q(t)\in S$ with $q(0)=q_0$ and $\dot q(0)\in T_{q_0}S$, one has
  \begin{equation}\label{eqn: stationary condition}
      \left.\frac{\mathrm d}{\mathrm dt}\right|_{t=0} f_{\xi,\eta}(q(t)) = 0.
  \end{equation}
\end{itemize}
\end{definition}
Given $\zeta \in \mathfrak g$, the locked motion at configuration $q\in Q$ is the curve
\[
q(t)=\exp(t\zeta)\cdot q, \qquad \dot q(0)=\zeta_Q(q),
\]
i.e. motion along the $G$-orbit with infinitesimal velocity $\zeta$. Its kinetic energy is given by
\begin{equation}\label{eqn: kinetic energy locked inertia}
T_{\mathrm{lock}}(q;\zeta):=\frac12\,g_q\!\big(\zeta_Q(q),\zeta_Q(q)\big)
=\frac12\,\langle \mathbb I(q)\zeta,\zeta\rangle
=\frac12\,f_{\zeta,\zeta}(q).
\end{equation}
Then, stationarity of the locked inertia tensor along $K$ at $q_0$ means that, if one perturbs the configuration transversely to the $K$-orbit (within the slice $S$) while keeping $\zeta$ fixed, this ``locked kinetic energy'' $T_\mathrm{lock}$ does not change up to first order at $q_0$.
\bigskip
\begin{maintheorem}\label{thm:drift-free-mech}
Let $(Q,g,V,G)$ be a simple mechanical system with symmetry. Let $H:T^*Q\to\mathds R$ be the associated $G$-invariant Hamiltonian, and let $J:T^*Q\to\mathfrak g^*$ be the momentum map of the cotangent-lifted $G$-action. Suppose $x_0=(q_0,p_0)\in T^*Q$ is a relative equilibrium of $H$ with velocity
$\xi_0\in\mathfrak g$, i.e.\ $x_0$ is a critical point of the augmented
Hamiltonian 
\begin{equation}\label{eqn: augmented hamiltonian}
H_{\xi_0}:=H-\langle J,\xi_0\rangle.
\end{equation}
Let $K\subset G_{\xi_0}$ be a compact subgroup with Lie algebra $\mathfrak k$, let $J_K:T^*Q\to\mathfrak k^*$ be the corresponding momentum map, and denote by
\begin{equation}\label{eqn: JKnu}
\nu:=J_K(x_0),\quad J_K^\nu:=J_K^{-1}(\nu).
\end{equation}
Assume that the $G$-action on $Q$ is proper and locally free in a neighborhood of $q_0$, and that the associated locked inertia tensor is stationary along $K$ (see Definition~\ref{def:simple-mech-stationary}). Then, in a neighborhood of $x_0$, the ($K_\nu$-invariant) quadratic Hamiltonian $H_{2,\xi_0}$ is drift-free near $x_0$ along $J^\nu_K$ with respect to the $K_\nu$-action (see Definition~\ref{def:drift-free}).
\end{maintheorem}
\begin{remark}
The quadratic part $H_2$ (of $H$) and $H_{2,\xi_0}$ (of $H_{\xi_0}$) coincide. Hence the drift-free condition of $H_{2,\xi_0}$ implies the drift-free condition of $H_2$.
\end{remark}

\subsection{Limiting singular regimes}\label{subsec:limiting}

Theorems~\ref{thm:general_intertwining_JS},~\ref{thm:main-BNF} and~\ref{thm:drift-free-mech} are formulated for dynamics on a smooth momentum level $J_\mu$ (see Remark~\ref{remark:coisotropic}), modeled locally by a second-class constraint system (see Definition~\ref{def:first-second-class}) defining a transverse symplectic slice, and intertwined with the reduced dynamics on a symplectic stratum $\mathcal S\subset M_\mu$. In particular, the hypotheses require that:
\begin{itemize}
    \item the relevant momentum level is a smooth submanifold, so that the restricted presymplectic form and its
    characteristic distribution are defined in the usual sense;
    \item  locally, one can represent the quotient by a transverse slice defined by second-class constraints.
\end{itemize}
In this section we record two examples that motivate the present work but fall outside this framework, and we indicate precisely which part of the mechanism breaks down.

\paragraph{Kustaanheimo-Stiefel regularization near the origin.}
We briefly recall the geometric framework of the Kustaanheimo-Stiefel (KS) regularization in the form relevant to the present discussion (for a detailed account we refer to~\cite{Zhao2015}). Let $\mathbb H$ denote the algebra of quaternions, written as $\mathbb H=\{z_0+z_1 i+z_2 j+z_3 k: z_i\in\mathds R\}$, with conjugation $\bar z=z_0-z_1 i-z_2 j-z_3 k$ and norm $|z|=\sqrt{z\bar z}$. Let $\mathbb I\mathbb H:=\{z\in\mathbb H: \mathrm{Re}(z)=0\}\simeq\mathds R^3$ be the purely imaginary quaternions. Fix the unit imaginary quaternion $i\in\mathbb I\mathbb H$ and define the Hopf map
\begin{equation}\label{eq:hopf_limiting}
\pi:\mathbb H\longrightarrow \mathbb I\mathbb H,\qquad \pi(z)=zi\bar z .
\end{equation}
The map $\pi$ is homogeneous of degree two and is constant along the orbits of the circle action $z\mapsto e^{i\theta}z$. In $\mathbb H\setminus\{0\}$ this action is free and $\pi$ realizes $\mathbb H\setminus\{0\}$ as a principal $S^1$-bundle over $\mathbb I\mathbb H\setminus\{0\}$.

Consider now the cotangent bundle $T^*(\mathbb H\setminus\{0\})\simeq \mathbb H\setminus\{0\}\times \mathbb H$ with canonical symplectic form, which is written in quaternionic notation as
\begin{equation}\label{eq:omegaKS_limiting}
  \omega_{\mathrm{KS}}=\mathrm{Re}\!\left(d\bar w\wedge dz\right).
\end{equation}
The diagonal circle action
\begin{equation}\label{eq:S1KS_limiting}
  e^{i\theta}\cdot(z,w)=(e^{i\theta}z,e^{i\theta}w)
\end{equation}
is Hamiltonian. A momentum map is given by the \emph{bilinear relation}
\begin{equation}\label{eq:BL_limiting}
  \mathrm{BL}(z,w)=\mathrm{Re}(\bar ziw)\in\mathds R,
\end{equation}
whose Hamiltonian vector field is the fundamental vector field of the action. Since the action is free on $\mathbb H\setminus\{0\}$, the value $0$ is a regular value of $\mathrm{BL}$ and the constraint surface
\[
  \Sigma_0:=\mathrm{BL}^{-1}(0)\subset T^*(\mathbb H\setminus\{0\})
\]
is a smooth coisotropic hypersurface. Its characteristic distribution is generated by the $S^1$-orbits, and
the quotient $V_0:=\Sigma_0/S^1$ is a smooth symplectic manifold. In the usual KS
construction, $V_0$ is identified with the physical phase space away from collision by means of the KS map
\[
  \psi:\Sigma_0\to T^*(\mathbb{IH}\setminus\{0\}),
  \qquad
  \psi(z,w)=\left(\pi(z), \frac{zi\bar w}{2|z|^2}\right),
\]
which factors through the quotient because its differential has kernel equal to the $S^1$-direction.

Since the circle action is free, and $\Sigma_0=\mathrm{BL}^{-1}(0)$ is a smooth coisotropic hypersurface, near any point $(z_*,w_*)\in\Sigma_0$ with $z_*\neq 0$, one may choose a local coordinate along the $S^1$-orbits and a complementary submanifold $N\subset\Sigma_0$ transverse to the orbits. Locally, $N$ is obtained by imposing one additional condition selecting a unique representative in each nearby $S^1$-orbit (see Section~\ref{sec: Proof of Thms} for the details). Together with the condition $\mathrm{BL}=0$ this gives a system of second-class constraints (see Definition~\ref{def:first-second-class}), and the associated Dirac bracket agrees with the intrinsic Poisson bracket on $N$. Therefore, away from $z=0$ the KS regularization fits naturally into the framework of smooth Hamiltonian reduction and local Dirac brackets.

Now consider the origin $(z,w)= (0,0)$, which is an equilibrium of the extended regularized system. This point is excluded from $\mathbb H\setminus\{0\}$, and therefore it lies outside the domain where the above construction is smooth. If one artificially enlarges the ambient space to $T^*\mathbb H$ to include $(z,w)=(0,0)$, then the hypotheses underlying the comparison theorems fail at that point for two independent reasons:
\begin{itemize}
\item the diagonal $S^1$-action has full isotropy at $(0,0)$ (a fixed point), so the stabilizer subgroup changes from $\{1\}$ to $S^1$ and therefore the orbit of $(0,0)$ is singular in the quotient;
\item the momentum map $\mathrm{BL}(z,w)$ is quadratic, hence $d(\mathrm{BL})_{(0,0)}=0$, so the constraint
$\mathrm{BL}=0$ is not regular near the origin and $\mathrm{BL}^{-1}(0)$ is not a smooth coisotropic hypersurface there.
\end{itemize}
As a consequence, the ``second-class slice'' mechanism breaks down at $(0,0)$ and thus the smooth Dirac-slice framework does not extend through the corresponding stratum.

Since $(0,0)$ does not correspond to a physical state, this is harmless in many applications. However this makes the use of a Birkhoff normal form procedure at $(0,0)$ to study physically relevant dynamics a difficult task.

\paragraph{Relative equilibria for the double spherical pendulum at $\mu=0$.}
We next describe the limiting regime arising in the double spherical pendulum example treated in Section~\ref{sec: double spherical pendulum}. The configuration space is $Q=S^2\times S^2$, where $q_1,q_2\in S^2$ represent the orientations of the two links. The phase space is $T^*Q$ with its canonical symplectic form. The system admits a natural $S^1$ symmetry given by simultaneous rotations about the vertical axis $e_3$:
\[
  e^{i\theta}\cdot(q_1,q_2)=(R_\theta q_1, R_\theta q_2),
\]
and this action lifts canonically to $T^*Q$. The associated momentum map is the component of the total angular momentum in the $e_3$-direction. For each momentum value $\mu\in\mathfrak{so}(2)^*\simeq\mathds R$ one may therefore consider the momentum level $J^{-1}(\mu)$ and the corresponding reduced space $J^{-1}(\mu)/S^1$, which in general is a stratified symplectic space.

Among the symmetric configurations are the \emph{straight} (or \emph{vertical}) configurations, where each link is aligned with the vertical axis:
\[
  q_1=\sigma_1 e_3,\qquad q_2=\sigma_2 e_3,\qquad \sigma_1,\sigma_2\in\{+1,-1\}.
\]
The configuration $q_1=q_2=-e_3$ is referred to informally as the ``straight-down'' configuration, while $q_1=q_2=+e_3$ is ``straight-up'' (the mixed-sign cases correspond to one link up and the other down). These configurations are fixed by the $S^1$-action, hence their isotropy is the full circle. They play a distinguished role dynamically: depending on parameters and on the choice of momentum value, they may correspond to equilibria or relative equilibria of the full Hamiltonian system (see Section~\ref{sec: double spherical pendulum} and Figure~\ref{fig:double_sphere_equilibria} for the geometric picture and explicit discussions).

The limiting regime relevant here occurs when $\mu=0$. At this momentum value, the reduced space contains the images of the fixed points above as isolated singular points, since the corresponding orbits have larger isotropy than nearby ones. Consequently, the assumptions needed for a smooth Dirac-slice model (see Lemma~\ref{lemma:Dirac bracket}) break down at the fixed point itself: near such a singular point of the quotient, the reduced space is not modeled by a smooth transverse ``slice'' arising from a second-class constraint pair with invertible constraint bracket matrix.

\section{Dirac brackets on momentum levels and comparison with reduction}\label{sec: Proof of Thms}
This section is devoted to prove Theorems~\ref{thm:general_intertwining_JS} and~\ref{thm:main-BNF}. The proof of Theorem~\ref{thm:general_intertwining_JS} is based on a local construction of second-class constraints adapted to the momentum level $J_\mu = J^{-1}(\mu)$ and a comparison between the ambient Poisson bracket and the Dirac bracket along $J_\mu$. 

To prove Theorem~\ref{thm:general_intertwining_JS}, our argument is structured as follows:
\begin{enumerate}
\item We construct the local geometry near $x_0$: we choose a slice inside the momentum level (see Definition~\ref{def: local slice}), describe its transversality with respect to the group orbits, and extend it to a transverse submanifold of the ambient symplectic manifold.

\item Using this geometric setup, we introduce a local second-class constraint map (see Definition~\ref{def:first-second-class}) whose zero set is precisely the local slice. This provides the Dirac-bracket model that will be used in the comparison.

\item We then compare the Dirac-bracket dynamics with the dynamics induced on the momentum level. The key point is that, because the Hamiltonian is invariant and drift-free along the chosen slice (see Definition~\ref{def:drift-free}), the Dirac correction terms vanish on the slice, so the two induced derivations agree there.

\item Finally, we extend this equality from the slice to a $G_\mu$-invariant neighborhood in the momentum level and then pass to the symplectic stratum of the quotient. This yields both~\eqref{eqn: LD =LJ} and~\eqref{eqn: main-intertwining-F}, completing the proof of Theorem~\ref{thm:general_intertwining_JS}.
\end{enumerate}

\begin{proof}[Proof of Theorem~\ref{thm:general_intertwining_JS}] 

Since $J$ is equivariant, the coadjoint isotropy group $G_\mu$ in~\eqref{eqn: notation J, Gmu, Mmu} is the stabilizer of $\mu$ for the coadjoint action, hence it is a closed subgroup of $G$. By Cartan's closed subgroup theorem~\cite{Cartan1930GroupesAnalysisSitus}, $G_\mu$ is an embedded Lie subgroup with Lie algebra $\mathfrak g_\mu$. Moreover, since the Hamiltonian $G$-action on $M$ is proper, the restricted $G_\mu$-action and in particular its restriction to $J_\mu$ is also proper. Then the orbit $\mathcal O := G_\mu \cdot x_0$ is an embedded submanifold of $J_\mu$, and therefore 
\begin{equation}\label{eqn: dim l}
\ell := \dim \mathcal O=\dim T_{x_0} \mathcal O =\dim\mathfrak g_\mu-\dim\mathfrak g_{x_0},
\end{equation}
where $\mathfrak g_{x_0}$ denotes the Lie algebra of the stabilizer of $x_0$ within $G_\mu$.
Choose a linear complement $\mathfrak l\subset\mathfrak g_\mu$ of $\mathfrak g_{x_0}$, so that $\mathfrak g_\mu=\mathfrak g_{x_0}\oplus \mathfrak l$, and fix a basis $\{\xi_i\}_{i=1}^\ell$ of $\mathfrak l$. Define
\begin{equation}\label{eqn: Phi momentum map}
\Phi_i(x):=\langle J(x)-\mu,\xi_i\rangle,\qquad i=1,\dots,\ell.
\end{equation}
\paragraph{Step 1. Construction of the local geometry.}
By the slice theorem for proper actions~\cite{palaisSlice} (see also~\cite{ortega2013momentum}), there exists a $G_{x_0}$-invariant slice
$C\subset J_\mu$ through $x_0$ (see Definition~\ref{def: local slice}), contained in $\pi_\mu^{-1}(\mathcal S)$ (see~\eqref{eqn: quotient map}), such that
\begin{equation}\label{eqn: TxJmu decomposition}
T_xJ_\mu = T_x(G_\mu\cdot x)\oplus T_x C,\quad \forall x\in C;
\end{equation}
and whose saturation $G_\mu\cdot C$ is an open neighborhood of $x_0$ in $J_\mu$. Therefore 
\begin{equation}\label{eqn: span Gmu-orbit}
T_x(G_\mu \cdot x) = \mathrm{span}\{X_{\Phi_i}(x)\}_{i=1}^\ell.\end{equation}
By hypothesis, $F\in C^\infty(M)^G$ is drift-free near $x_0\in J_\mu$ (see Definition~\ref{def:drift-free}). Using the decomposition~\eqref{eqn: TxJmu decomposition} on the slice $C\subset \pi_{\mu}^{-1}(\mathcal S)$, we have
\begin{equation}\label{eqn: invariance of V}
X_F(x)\in T_xC,\qquad \text{for all }x\in C.
\end{equation}
Therefore $C$ is invariant under the local flow of $X_F$ in $J_\mu$ near $x_0$.

We next extend $C$ to a submanifold $\Sigma \subset M$ with
\begin{equation}\label{eqn: Sigma}
  \Sigma \cap J_\mu = C\quad \text{and}\quad T_x\Sigma \cap T_x J_\mu = T_xC
\end{equation}
for $x$ near $x_0$. To construct $\Sigma$, fix a Riemannian metric on $M$ and let $W_x \subset T_xM$ be the orthogonal complement of $T_xJ_\mu$ for $x$ near $x_0$. Then $W := \underset{x}{\bigsqcup}\; W_x$ is a smooth subbundle of $TM$ along $J_\mu$. Using the Riemannian exponential map $\exp \colon TM\to M$, we define
\[
  \Sigma := \left\{ \exp_x(w) : x \in C,\ w \in W_x,\ \|w\| < \varepsilon \right\}.
\]
Here $\varepsilon>0$ is chosen sufficiently small so that $\exp_x$ is defined on $B_\varepsilon(0)\subset T_xM$ for all $x\in C$ and $\Sigma$ is an embedded submanifold. Moreover, since $\mathrm d(\exp_x)_0 = \operatorname{id}_{T_xM}$ and $T_xM = T_xJ_\mu \oplus W_x$, we have
\[
  T_x\Sigma = T_xC \oplus W_x
\]
and hence $T_x\Sigma \cap T_xJ_\mu = T_xC$ for all $x \in C$ near $x_0$, obtaining \eqref{eqn: Sigma}. Then $\Sigma $ intersects $J_\mu$ cleanly and is transverse to the $G_\mu$-orbits near $x_0$. 

\paragraph{Step 2. Construction of the second-class constraint.}
Since $\Sigma$ is a  $\ell$-codimensional submanifold of $M$, we can describe it locally as the common zero set of $\ell$ smooth functions. Let $\dim M =:2m$ and choose an open neighborhood $Z \subset M$ of $x_0$ together with a coordinate chart $\kappa : Z \to B \subset \mathds R^{2m}$ such that $\kappa(x_0) = 0$. Write
\[
  \kappa(x) = (r(x), s(x)), \qquad
  r(x) \in \mathds R^\ell,\ s(x) \in \mathds R^{2m-\ell}.
\]
The inclusion $\iota : \Sigma \cap Z \hookrightarrow Z$ is an immersion of rank $2m-\ell$, so the composition
\[
  \kappa \circ \iota : \Sigma \cap Z \longrightarrow B
\]
has constant rank $2m-\ell$ near $x_0$. By the constant rank theorem, we may choose the coordinates $(r,s)$ on $B$ so that
\[
  \kappa(\Sigma \cap Z) = \{ (r,s) \in B : r = 0 \}.
\]
Writing $r(x) = (r_1(x),\dots,r_\ell(x))$, we denote by
\[
  \Upsilon_i(x) := r_i(x), \qquad i = 1,\dots,\ell,
\]
and set $\Upsilon = (\Upsilon_1,\dots,\Upsilon_\ell) : Z \to \mathds  R^\ell$. By construction, $\mathrm d\Upsilon$ has rank $\ell$ along $\Sigma \cap Z$ and
\begin{equation}\label{eqn: chi=0}
  \Sigma \cap Z = \{ x \in Z \colon \Upsilon_1(x) = \cdots = \Upsilon_\ell(x) = 0 \},
\end{equation}
whose differentials $\mathrm d\Upsilon_1, \cdots, \mathrm d \Upsilon_\ell$ are linearly independent along $\Sigma \cap Z$.

We set
\begin{equation}\label{eqn: U}
U:= J_\mu \cap Z =\left\{x\in Z\colon \Phi_1(x) = \cdots=\Phi_\ell(x)= 0\right\}
\end{equation}
where $\Phi=(\Phi_1,\cdots,\Phi_\ell)$ is defined in~\eqref{eqn: Phi momentum map} and
\begin{equation}\label{eqn: N}
N:=C \cap Z = \Sigma \cap U = \{x\in U\colon\Phi_1(x)=\cdots=\Phi_{\ell}(x)=\Upsilon_1(x)=\cdots=\Upsilon_{\ell}(x) = 0\}.
\end{equation}
For each $x\in N$, since $U=\{\Phi=0\}$ and $\Sigma\cap Z=\{\Upsilon=0\}$ are regular level sets, we have
\begin{equation}\label{eqn: TxSigma kernel}
T_xU=\bigcap_{i=1}^\ell \ker \mathrm d\Phi_i(x),
\qquad
T_x\Sigma=\bigcap_{j=1}^\ell \ker \mathrm d\Upsilon_j(x).
\end{equation}
Hence
\[
\ker \mathrm d\varphi(x) = \left(\bigcap_{i=1}^\ell \ker \mathrm d\Phi_i(x)\right)\cap \left(\bigcap_{j=1}^\ell \ker \mathrm d\Upsilon_j(x)\right) = T_xU\cap T_x\Sigma = T_xN.
\]
Since \(N\) is a smooth codimension-\(2\ell\) submanifold, \(\dim T_xN=\dim M-2\ell\), and therefore
\(\operatorname{rank}\mathrm d\varphi(x)=2\ell\) for all \(x\in N\).
Thus the map 
\[\varphi = (\Phi,\Upsilon)\colon Z\subset M \to \mathds R^{2\ell}\]
has constant rank \(2\ell\) along \(N\).

Moreover, $U\subset J_\mu$ is coisotropic in $(M,\omega)$ (see Remark~\ref{remark:coisotropic}), and applying~\eqref{eqn: span Gmu-orbit} its characteristic distribution is
\[
(T_xU)^\omega = T_x(G_\mu\cdot x)=\operatorname{span}\{X_{\Phi_1}(x),\dots,X_{\Phi_\ell}(x)\},
\qquad x\in U.
\]
Since $\Sigma$ is transverse to this characteristic distribution (see~\eqref{eqn: TxJmu decomposition} and~\eqref{eqn: Sigma}), no nontrivial linear combination of the $X_{\Phi_i}(x)$'s lies in $T_x\Sigma$. Using~\eqref{eqn: TxSigma kernel}, this is equivalent to the invertibility of the $\ell\times \ell$ matrix
\[
B(x):=\left(\mathrm d\Upsilon_j(x)\,X_{\Phi_i}(x)\right)_{i,j}=\left(\{\Upsilon_j,\Phi_i\}_M(x)\right)_{i,j}, \qquad x\in N.
\]
On the other hand, since the functions $\Phi_i$ are the components of the momentum map corresponding to $\mathfrak g_\mu$, one has
\[
\{\Phi_i,\Phi_j\}_M=0 \qquad \text{on } U.
\]
Therefore, the constraint matrix of $\varphi=(\Phi,\Upsilon)$ along $N$ has the block form
\[
C_\varphi(x)= \begin{pmatrix}0 & B(x)\\ - B(x)^T & * \end{pmatrix}, \qquad x\in N,
\]
and is invertible because $B(x)$ is invertible. Thus $\varphi$ is a local second-class constraint (in the sense of Definition~\ref{def:first-second-class}) defining $N$ (see Figure~\ref{fig:second_class_constraint}).

\begin{figure}[ht!]
\centering
    \begin{subfigure}{0.48\textwidth}
       \centering
        \begin{overpic}[width=\linewidth]{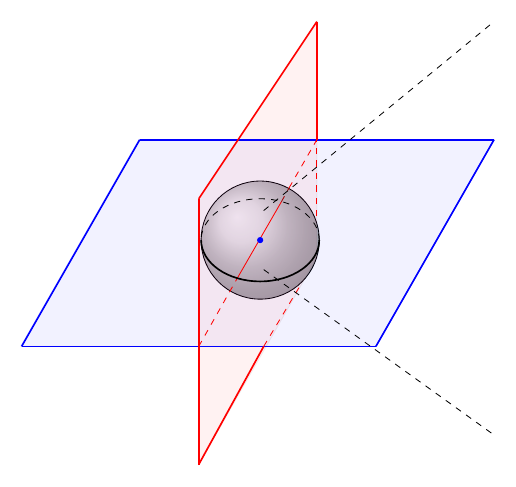}
            \put(63,85){\color{red}$ \Sigma$}
            \put(87,63){\color{blue}$J_\mu$}
            \put (52,48){\color{blue} $x_0$}
            \put (63,45){$Z\subset M$}
            \put (41,30){\color{red}$\Sigma\cap Z:=\{\Upsilon= 0\}$}
        \end{overpic}
    \end{subfigure}
    \hspace{-0.7cm}
    \begin{subfigure}{0.48\textwidth}
        \centering
        \begin{overpic}[width=\linewidth]{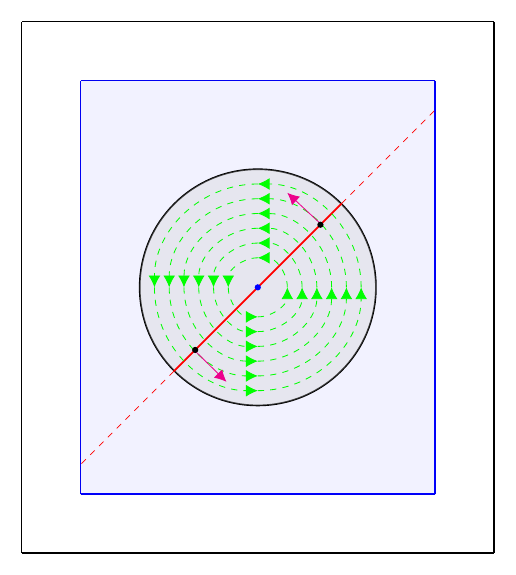}
            \put(17,80){\color{blue} $J_\mu$}
            \put(35,73){$U=J_\mu \cap Z$}
            \put(41,51.5){\color{blue}$x_0$}
            \put(38,40){\color{red} $N = \Sigma\cap U$}
            \put(37,25){\color{green} $G_\mu$-orbits}
            \put(18,19){\color{red!60}  $C = \Sigma \cap J_\mu$}
        \end{overpic}
    \end{subfigure}
    \caption{Geometry of the second-class constraints. On the left, the momentum level $J_\mu$ (in blue), a neighborhood $Z\subset M$ of the relative equilibrium $x_0$ (in grey), and the transverse hypersurface $\Sigma$ (in red). The intersection $\Sigma\cap Z$ is a plane parameterized by $\{\Upsilon= 0\}$. On the right, a view inside the neighborhood $U=J_\mu \cap Z$ (grey disc). It is foliated by the $G_\mu$-orbits, which are the integral curves of the Hamiltonian vector field $X_{\Phi}$ related to the first-class constraint $\Phi$ (green dashed circles). Intersecting with $\Sigma$ we have the slice $N=\Sigma\cap U = \{\Phi=0,\Upsilon=0\}$ (red segment) through $x_0$, meeting each nearby orbit. Thus the pair of constraint functions $(\Phi,\Upsilon)$ parameterizes the directions along and transverse to the characteristic foliation and defines the second-class constraint.}
    \label{fig:second_class_constraint}
\end{figure}
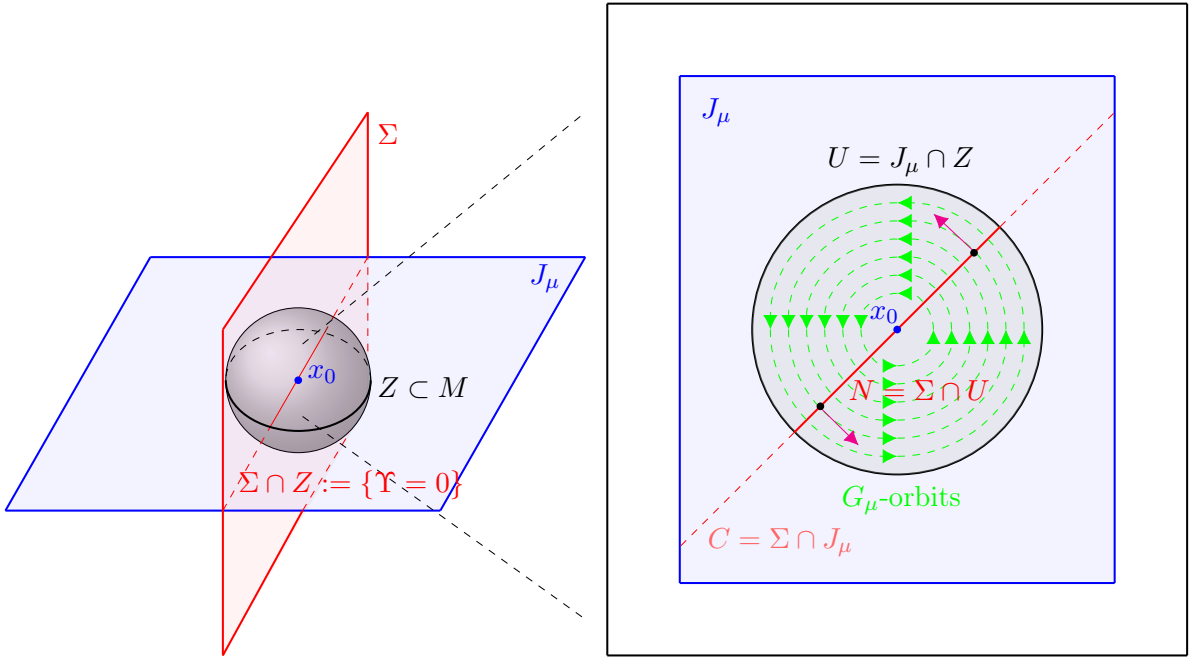

\paragraph{Step 3: Comparison of the Dirac and induced dynamics on the slice $N$.}
Since $F$ is $G$-invariant and the Hamiltonian vector fields $X_{\Phi_i}$ generate the $G_\mu$-action on $M$, we have
\begin{equation}\label{eqn: F-Phi}
  \{F,\Phi_i\}_M = 0 \quad \text{on } Z,\qquad i=1,\dots,\ell.
\end{equation}
From \eqref{eqn: invariance of V} and \eqref{eqn: chi=0} we have that $X_F(x) \in T_xC \subset T_x\Sigma$. Using~\eqref{eqn: TxSigma kernel} we have
\begin{equation}\label{eqn: F-chi}
d\Upsilon_i(x)\left(X_F(x)\right) = \{\Upsilon_i,F\}_M=0\quad \text{for all } x\in C,\qquad i=1,\cdots, \ell.
\end{equation}
In particular, the same holds on $N = C\cap Z\subset C$. We apply Lemma~\ref{lemma:Dirac bracket} to the local second-class constraint $\varphi = (\Phi,\Upsilon)$ defining $N$ so that for each $H\in C^\infty(M)$ and each $x\in N$ we have
\[\{F,H\}_D(x) = \{F,H\}_M - \sum_{i,j=1}^{2\ell} \{F,\varphi_i\}_M(x)\, C^{-1}_{ij}(x)\,\{\varphi_j,H\}_M(x),\]
where $C_{ij} := \{\varphi_i,\varphi_j\}_M$ and $C^{-1}$ is its inverse along $N$. Thanks to \eqref{eqn: F-Phi} and \eqref{eqn: F-chi}, all terms in the correction vanish on $N$, and therefore $\{F,H\}_D = \{F,H\}_M$ on $N$. 
In particular, the corresponding Hamiltonian vector fields of $F|_{J_\mu}$ satisfy 
\begin{equation}\label{eqn: XF-equality-N}
X^F_D = X^F_{J_\mu}\qquad \text{on } N,
\end{equation}
where  $X^F_D$ and $X^F_{J_\mu}$ denote the Hamiltonian vector fields of $F|_{J_\mu}$ with respect to $\{\cdot,\cdot\}_D$ and $\{\cdot,\cdot\}_{J_\mu}$, respectively (see~\eqref{eq:dirac-bracket} and~\eqref{eqn:induced-bracket-J}). 

\paragraph{Step 4: Extension to $U$ and passage to the reduced stratum.}
Note that both $X_D^F$ and $X_{J_\mu}^F$ are $G_\mu$-equivariant on $J_\mu$. After possibly shrinking \(Z\), we may assume that $U$ in~\eqref{eqn: U} is $G_\mu$-invariant and contained in the saturation $G_\mu\cdot C$. By transversality of $\Sigma$ to the $G_\mu$-orbits, the slice $N=\Sigma\cap U$ intersects every $G_\mu$-orbit in $U$, and therefore 
\begin{equation}\label{eqn: neighborhood U}
U=G_\mu\cdot N\subset G_\mu\cdot C.
\end{equation}
For all $x\in U$, there exist $g_\mu\in G_\mu$ and $n\in N$ with $x=g_\mu\cdot n$. Using $G_\mu$-equivariance and~\eqref{eqn: XF-equality-N}, we obtain
\[X_D^F(x) = (\Phi_{g_\mu})_*X_D^F(n) = (\Phi_{g_\mu})_*X_{J_\mu}^F(n) = X_{J_\mu}^F(x),\]
where $\Phi_{g_\mu}$ denotes the associated $g_\mu$-action. Hence $X_D^F=X_{J_\mu}^F$ on $U$ and therefore the following diagram commutes
\begin{equation}\label{diag:15-clean}
\begin{tikzcd}[column sep=large,row sep=large]
C^\infty(Z)
  \arrow[r,"\mathcal L_D^F"]
  \arrow[d,"r_U"']
&
C^\infty(Z)
  \arrow[d,"r_U"]
\\
C^\infty(U)
  \arrow[r,"\mathcal L_{J_\mu}^F"]
&
C^\infty(U)
\arrow[ul, phantom, "{\scalebox{1.4}{$\circlearrowleft$}}" description]
\end{tikzcd}
\end{equation}
yielding~\eqref{eqn: LD =LJ}. Here $\mathcal L^F_D$ and $\mathcal L^F_{J_\mu}$ are defined in \eqref{eqn: linoperators}, and $r_U$ is the restriction map $r_U(H)=H|_U$. 


By singular symplectic reduction theory~\cite{sjarmaanlerman}, the restriction of $\pi_{\mu}$ to $\pi_{\mu}^{-1}(\mathcal S)$ is a Poisson map. Namely, for any $f,h\in C^\infty(\mathcal S)$ one has
\[
\pi_{\mu}^*\{f,h\}_{\mathcal S}=\{\pi_{\mu}^*f,\pi_{\mu}^*h\}_{J_\mu},
\]
where the pullbacks $\pi_{\mu}^*f,\pi_{\mu}^*h\in C_{\mathrm{adm}}^\infty(J_\mu)$ in~\eqref{eqn: admissible functions}. 

The $G$-invariance of $F$ yields $F|_{J_\mu}=\pi_{\mu}^*\tilde F$, where $\tilde F$ is defined in~\eqref{eqn: linoperators}. Set 
\begin{equation}\label{eqn: neighborhood V}
V:=\pi_{\mu}(U)\cap \mathcal S.
\end{equation}
Using~\eqref{eqn: linoperators} and~\eqref{eqn: LD =LJ}, for any $g\in C^\infty(V)$ on $U\cap \pi_{\mu}^{-1}(V)$ we obtain
\begin{equation*}
\pi_{\mu}^*\left(\mathcal L_{\mathcal S}^F g\right)=\pi_{\mu}^*\{\tilde F,g\}_{\mathcal S}=\{\pi_{\mu}^*\tilde F,\pi_{\mu}^*g\}_{J_\mu}=\mathcal L_{J_\mu}^F(\pi_{\mu}^*g)=\mathcal L_D^F(\pi_{\mu}^*g)
\end{equation*}
yielding~\eqref{eqn: main-intertwining-F}. This completes the proof of Theorem~\ref{thm:general_intertwining_JS}.
\end{proof}

The proof of Theorem~\ref{thm:main-BNF} is based on a combination of Theorem~\ref{thm:general_intertwining_JS} applied to the quadratic part of a symmetric Hamiltonian at a relative equilibrium with the standard Birkhoff normal form procedure on $J_\mu$.

\begin{proof}[Proof of Theorem~\ref{thm:main-BNF}]
Let $H\in C^\infty(M)^G$ be a $G$-invariant Hamiltonian, $x_0\in J_\mu$ be a relative equilibrium, and $H_2$ denote the quadratic part of $H$. By hypothesis, $H_2$ is drift-free near $x_0\in J_\mu$ (see Definition~\ref{def:drift-free}). Hence we apply Theorem~\ref{thm:general_intertwining_JS} with $F = H_2$ to obtain neighborhoods $U \subset J_\mu$ and $V \subset \mathcal S$ defined in~\eqref{eqn: U} and~\eqref{eqn: neighborhood V} respectively, such that
\begin{equation}\label{eqn:intertwining-H2}
  \pi_{\mu}^* \circ \mathcal L_S =\mathcal L_{J_\mu} \circ \pi_{\mu}^* \quad\text{on } C^\infty(V),
\end{equation}
where $\mathcal L_{J_\mu}, \mathcal L_S$ are defined in~\eqref{eqn: LJ and LS H2}. 

We now recall the standard Birkhoff normal form construction applied on the momentum level $J_\mu$ with respect to the operator $\mathcal L_{J_\mu}$ (see for instance \cite{CIFTCI20141, van2017geometry}). Write the restriction of $H$ to $J_\mu$ as a formal sum of homogeneous terms
\[
H|_{J_\mu} = H_2 + H_3 + H_4 + \cdots,
\]
where $H_k$ is homogeneous of degree $k$ in a local Darboux chart on a symplectic slice through $x_0$. Via this identification we regard $H|_{J_\mu}$ as a formal power series in the corresponding symplectic coordinates, and $\mathcal L_{J_\mu}=\{H_2,\cdot\}_{J_\mu}$ is computed with respect to the induced bracket on $J_\mu$ (which restricts to the usual Poisson bracket on the slice).

In each degree $k\ge 3$ we decompose
\[
H_k = H_k^{\mathrm{res}} + H_k^{\mathrm{nr}}
\]
where $H_k^{\mathrm{res}}\in\ker\mathcal L_{J_\mu}$ is the resonant part and $H_k^{\mathrm{nr}}\in\operatorname{im} \mathcal L_{J_\mu}$ is the nonresonant part with respect to $\mathcal L_{J_\mu}$. To solve the homological equation
\[
\mathcal L_{J_\mu}\Gamma_k = H_k^{\mathrm{nr}},
\]
we use the time-one flow of $\Gamma_k$ to remove $H_k^{\mathrm{nr}}$ from the Hamiltonian (up to terms of degree $>k$). Iterating this procedure for $k=3,4,\dots$ we obtain an a priori formal (or convergent, when applicable) canonical transformation on $J_\mu$ which conjugates $H$ to a Hamiltonian of the form
\[
 H_{\mathrm{BNF}} = H_2 + \sum_{k\ge 3} H_k^{\mathrm{NF}}.
\]
By construction, $H_{\mathrm{BNF}}$ generates on $J_\mu$ the same quadratic homological operator $\mathcal L_{J_\mu}$ as that of $H_2$. The defining property of the Birkhoff normal form is expressed entirely in terms of $\mathcal L_{J_\mu}$: for each degree $k\ge 3$, the nonresonant part has been eliminated (lies in $\operatorname{im}\mathcal L_{J_\mu}$) and the remaining term lies in $\ker\mathcal L_{J_\mu}$.

We perform the normal form computation on $U_V:=U\cap \pi_{\mu}^{-1}(V)$ within the Poisson subalgebra $\pi_{\mu}^* C^\infty(V) \subset C_{\mathrm{adm}}^\infty(U_V)$. At each order we choose the generating function $\tilde \Gamma_k$ and define $\Gamma_k=\pi_{\mu}^*\tilde \Gamma_k$. This is possible because, by~\eqref{eqn:intertwining-H2}, the homological operator satisfies
\[\mathcal L_{J_\mu}(\pi_{\mu}^* g) = \pi_{\mu}^*(\mathcal L_{\mathcal S} g)\quad \text{on }U_V,\;\forall g\in C^\infty(V),\]
so the homological equation on $J_\mu$ for pullback terms is equivalent to the homological equation on $V\subset \mathcal S$. Consequently, the resulting normal form Hamiltonian satisfies 
\[H_{\mathrm{BNF}}= \pi^*_{\mu}\tilde H_{\mathrm{BNF}}\]
for a unique local function $\tilde H_{\mathrm{BNF}}\in C^\infty(V)$. 

Finally, again by~\eqref{eqn:intertwining-H2}, the Birkhoff normal form conditions for $H_{\mathrm{BNF}}$ with respect to $\mathcal L_{J_\mu}$ restricted to $\pi_{\mu}^* C^\infty(V)$ are equivalent to the usual Birkhoff normal form conditions for $\tilde H_{\mathrm{BNF}}$ with respect to $\mathcal L_\mathcal S$ on $V$: $\tilde H_{\mathrm{BNF}}$ has quadratic part $\tilde H_2$, and at each order $k\geq 3$ the nonresonant part (with respect to $\mathcal L_\mathcal S$) has been removed. This proves that $\tilde H_{\mathrm{BNF}}$ is a Birkhoff normal form for the reduced dynamics on $V\subset \mathcal S$, completing the proof.
\end{proof} 

\section{A criterion for reduced Birkhoff normal forms in simple mechanical systems}\label{sec: Proof of mechsys}
In this section we prove Theorem~\ref{thm:drift-free-mech}. The proof relies on the combination of the geometric structure of simple mechanical systems (see Definition~\ref{def:simple-mech-stationary}) with the Marle-Guillemin-Sternberg (MGS) normal form for Hamiltonian group actions~\cite{guillemin1984normal,marle1985modele} (for a modern formulation, we refer to \cite{ortega2013momentum,miguelRodriguez2014}), which we briefly recall here.

\begin{theorem}[MGS normal form]\label{thm: MGS}
Let $K$ be a compact Lie group acting properly and in a Hamiltonian fashion on a symplectic manifold $(M,\omega)$ with equivariant momentum map $J_K:M\to\mathfrak k^*$. Fix $x_0\in M$, set $\nu:=J_K(x_0)$, and denote by $H:=K_{x_0}$ the stabilizer of $x_0$ with Lie algebra $\mathfrak h$. 

Choose an $\Ad_H$-invariant splitting $\mathfrak k=\mathfrak h\oplus\mathfrak m$ (equivalently, identify $\mathfrak m^*$ with the annihilator of $\mathfrak h$ in $\mathfrak k^*$). Let $W$ be a symplectic slice at $x_0$, so that $W$ is a symplectic vector space carrying a linear symplectic $H$-action with (linear) slice momentum map 
\begin{equation}\label{eqn: slice momentum map}
J_W:W\to\mathfrak h^*.
\end{equation}
Then there exists a $K$-equivariant symplectomorphism
\begin{equation}\label{eqn: MGS symplectomorphism}
\Psi\colon U \subset M \to U_Y\subset Y
\end{equation}
from a neighborhood $U$ of $x_0$ onto a neighborhood $U_Y$ of the distinguished point $[e,0,0]$ in the associated bundle
\begin{equation}\label{eqn: MGS bundle}
Y := K\times_{H}(\mathfrak m^*\times W),
\end{equation}
In the induced bundle coordinates $[k,\eta,u]\in Y$, the momentum map takes the form
\begin{equation}\label{eqn: MGS momentum map}
\left(J_K\circ \Psi^{-1}\right)([k,\eta,u])=\Ad^*_{k^{-1}}\left(\nu+\eta+J_W(u)\right).
\end{equation}
\end{theorem}
To prove Theorem~\ref{thm:drift-free-mech}, our argument is structured as follows:
\begin{enumerate}
  \item We recall the vertical-horizontal decomposition of $TQ$ induced by the metric $g$. Then we express the momentum map in terms of the locked inertia tensor defined in~\eqref{eqn: locked inertia tensor}, and we rewrite the associated Hamiltonian accordingly.
  
  \item We provide an explicit formulation for the augmented Hamiltonian $H_{\xi_0}$ in \eqref{eqn: augmented hamiltonian}. We isolate its dependence on the momentum variables, and exploit the relative equilibrium condition to control its first derivatives at $x_0$.
  
  \item We apply Theorem~\ref{thm: MGS} for the $K$-action to construct local symplectic coordinates $(\theta,I,u)$ adapted to the $K$-orbits and to a symplectic slice, in which $\omega$ is in Darboux form, the momentum map $J_K$ is affine in $I$, and $H_{\xi_0}$ is independent of the group variables.
  
  \item We make use of the stationarity of the locked inertia tensor along $K$ (see Definition~\ref{def:simple-mech-stationary}) to obtain a block-diagonal structure of the Hessian of $H_{\xi_0}$ at $x_0$. Then we deduce that the quadratic part $H_{2,\xi_0}$ is drift-free near $x_0$ along $J_K^\nu$ in~\eqref{eqn: JKnu}, completing the proof.
\end{enumerate}

\paragraph{Step 1. Momentum map and Kaluza-Klein Hamiltonian.}
By definition, the Hamiltonian function associated to a simple mechanical system with symmetry is of the form
\begin{equation}\label{eqn: Hamiltonian}
  H(q,p) = \frac12\,g_q^{-1}(p,p) + V(q).
\end{equation}
For each $q\in Q$, the $G$-action defines the vertical subspace
\[
  V_q := T_q(G\cdot q) =  \{\xi_Q(q)\mid\xi\in\mathfrak g\}\subset T_qQ.
\]
Since $G$ acts by $g$-isometries, we define the horizontal subspace $H_q$ to be the $g$-orthogonal complement of $V_q$. Namely
\begin{equation}\label{eqn: orthogonal splitting}
  T_q Q = V_q\oplus H_q,\qquad g_q(V_q,H_q)=0.
\end{equation}
Thus every velocity $\dot q\in T_qQ$ decomposes uniquely as
\begin{equation}\label{eqn: orth-decomp-dotq}
  \dot q = \dot q_V + \dot q_H,\qquad
  \dot q_V\in V_q,\ \dot q_H\in H_q.
\end{equation}
At points where the action is locally free (in particular, in a neighborhood of the configuration $q_0$ of the relative equilibrium under consideration), the map
\[
  \mathfrak g \longrightarrow V_q,\qquad \xi \longmapsto \xi_Q(q)
\]
is a linear isomorphism. Hence there is a unique $\xi\in\mathfrak g$ such that $\dot q_V=\xi_Q(q)$, and we can write
\begin{equation}\label{eqn: dotq-chiq}
  \dot q = \xi_Q(q) + \dot q_H,\qquad  \text{in which }\dot q_H\perp V_q.
\end{equation}
The metric $g_q$ gives an isomorphism 
\[g_q\colon T_qQ\to T_q^*Q,\quad u_q\mapsto g_q(u_q,\cdot).\]
Hence, given $\dot q$ as in~\eqref{eqn: dotq-chiq}, we define
\begin{equation}\label{eqn: relation p metric}
  p := g_q(\dot q,\cdot) = g_q(\xi_Q(q),\cdot)+g_q(\dot q_H,\cdot) =: p_V + p_H \in T_q^* Q,
\end{equation}
where $p_V\in V_q^*$ and $p_H\in H_q^*$. Since $V_q$ and $H_q$ are $g_q$-orthogonal, the dual splitting $T_q^*Q=V_q^*\oplus H_q^*$ is orthogonal with respect to the dual metric $g_q^{-1}$, so
\begin{equation}\label{eqn: identities gq-1}
  g_q^{-1}(p_V,p_H)=0,\qquad
  g_q^{-1}(p_H,p_H)=g_q(\dot q_H,\dot q_H),\qquad g_q^{-1}(p_V,p_V) = g_q(\xi_Q(q),\xi_Q(q)) = \langle\mathds I(q)\xi,\xi\rangle,
\end{equation}
where $\mathds I(q)$ is the locked inertia tensor in~\eqref{eqn: locked inertia tensor}. By definition (see for instance~\cite[Sec.~6.6.27]{ortega2013momentum}), the momentum map $J:T^*Q\to\mathfrak g^*$ of the cotangent-lifted $G$-action is
given by
\[
  \langle J(q,p),\eta\rangle = p(\eta_Q(q)),\qquad \eta\in\mathfrak g.
\]
Applying~\eqref{eqn: relation p metric} and~\eqref{eqn: identities gq-1}, for $p=g_q(\dot q,\cdot)$ with $\dot q=\xi_Q(q)+\dot q_H$ we obtain
\[
  \langle J(q,p),\eta\rangle
  = g_q(\xi_Q(q),\eta_Q(q)) + g_q(\dot q_H,\eta_Q(q))
  = g_q(\xi_Q(q),\eta_Q(q))
  = \langle\mathds I(q)\xi,\eta\rangle,
\]
for all $\eta\in\mathfrak g$, since $\dot q_H$ is orthogonal to $V_q$. Hence
\begin{equation}\label{eqn: relation J,xi,I}
  J(q,p) = \mathds I(q)\xi,\qquad \xi = \mathds I(q)^{-1}J(q,p).
\end{equation}
Using the orthogonal splitting $p=p_V+p_H$ and~\eqref{eqn: identities gq-1} and~\eqref{eqn: relation J,xi,I}, the Hamiltonian~\eqref{eqn: Hamiltonian} has the ``Kaluza-Klein'' form 
\begin{equation}\label{eqn:KK-form}
 H(q,p)
  = \frac12\big\langle J(q,p),\mathds I(q)^{-1}J(q,p)\big\rangle
    + \frac12\,g_q^{-1}(p_H,p_H)
    + V(q).
\end{equation}
\begin{remark}
The terminology \emph{Kaluza-Klein} refers to the fact that the kinetic energy of a $G$-invariant metric on the total space of a (locally) free $G$-bundle can be written, using a connection, as the sum of a horizontal kinetic term and a vertical term quadratic in the momenta. In our setting, the orthogonal splitting~\eqref{eqn: orthogonal splitting} defines the mechanical connection and the associated Kaluza-Klein metric on $Q$. In cotangent variables this yields the decomposition \eqref{eqn:KK-form} of the Hamiltonian into a horizontal kinetic term plus the vertical quadratic term $\frac12\langle J, I(q)^{-1}J\rangle$ (see for instance \cite{Montgomery1990} and \cite{ortega2013momentum}).
\end{remark}

\paragraph{Step 2. Augmented Hamiltonian and relative equilibrium.}
For a relative equilibrium $x_0 = (q_0,p_0)$, fix $\mu=J(x_0)\in\mathfrak g^*$ and the velocity $\xi_0\in\mathfrak g$ satisfying~\eqref{eqn: relation J,xi,I}, that is
\begin{equation}\label{eqn: xi0}
\xi_0 = \mathds I (q_0)^{-1} J(x_0).
\end{equation}
Substituting~\eqref{eqn:KK-form} into the augmented Hamiltonian $H_{\xi_0}$ in~\eqref{eqn: augmented hamiltonian} gives
\begin{equation}\label{eqn:Hchi}
\begin{aligned}
  H_{\xi_0}(q,p) &= \frac12\langle J(q,p),\mathds I(q)^{-1}J(q,p)\rangle- \langle J(q,p),\xi_0\rangle + \frac12 g_q^{-1}(p_H,p_H) + V(q) \\
  &= \frac12\langle J(q,p)-\mathds I(q)\xi_0,\mathds I(q)^{-1}(J(q,p)-\mathds I(q)\xi_0)\rangle - \frac12\langle \mathds I(q)\xi_0,\xi_0\rangle + \frac12 g_q^{-1}(p_H,p_H) + V(q).
\end{aligned}
\end{equation}
At $x_0=(q_0,p_0)$ we have $J(x_0)-\mathds I(q_0)\xi_0=0$. Therefore the momentum-dependent term
\[K(q,J):=\langle J(q,p)-\mathds I(q)\xi_0,\mathds I(q)^{-1}(J(q,p)-\mathds I(q)\xi_0)\rangle\]
satisfies
\[
D_JK(q_0,J(x_0)) = \mathds I(q_0)^{-1}\left(J(x_0)-\mathds I(q_0)\xi_0\right) =0,
\]
or equivalently,
\[ \mathrm dK(q_0,J(x_0))(0,\delta J)=0 \qquad \forall\,\delta J\in\mathfrak g^*. \]
Since $x_0$ is a critical point of $H_{\xi_0}$, we have $\mathrm dH_{\xi_0}(x_0)=0$, so all first derivatives of $H_{\xi_0}$ vanish at $x_0$.

\paragraph{Step 3. Local coordinates adapted to the $K$-action and a symplectic slice.}
Since the $G$-action on $Q$ is proper and locally free near $q_0$, and $K\subset G_{\xi_0}$ is compact, the restricted $K$-action on $Q$ is also proper and locally free near $q_0$. In particular, the stabilizers $K_q$ are finite for $q$ close to $q_0$, so the orbit map $K \to K \cdot q_0$ is a finite covering.

By the slice theorem for proper actions~\cite{palaisSlice}, there exists a $K_{q_0}$-invariant slice $S\subset Q$ through $q_0$ such that $T_{q_0}Q = T_{q_0}(K\!\cdot\! q_0)\oplus T_{q_0}S$ and the saturation map
\begin{equation}\label{eqn: saturation map}
\widetilde\Phi:K\times S\longrightarrow U_Q\subset Q,\qquad (k,y)\longmapsto k\cdot y
\end{equation}
induces a $K$-equivariant diffeomorphism 
\begin{equation}\label{eqn: slice diffeomorphism}
\Phi:\ K\times_{K_{q_0}} S \ \longrightarrow\ U_Q,
\qquad [k,y]\longmapsto k\cdot y,
\end{equation}
onto a $K$-invariant neighborhood $U_Q$ of $K\cdot q_0$.

Consider now the cotangent lifted Hamiltonian $K$-action on $T^*Q$,
\[
G_K:K\times T^*Q\to T^*Q,\qquad (k,x)\mapsto k\cdot x,
\]
with momentum map $J_K:T^*Q\to\mathfrak k^*$, and set $\nu := J_K(x_0)$. Applying Theorem~\ref{thm: MGS} to the Hamiltonian $K$-action on $(T^*Q,\omega)$ at $x_0$, we obtain a $K$-equivariant symplectomorphism
\[
\Psi:U\subset T^*Q \longrightarrow U_Y\subset Y:=K\times_{K_{x_0}}(\mathfrak m^*\times W).
\]
Here $W$ is a symplectic slice at $x_0$ for the cotangent-lifted $K$-action and $\mathfrak m^*\cong (\mathfrak k/\mathfrak h)^*$, where $\mathfrak h$ is the Lie algebra of $K_{x_0}$. 

The $K$-action is locally free near $x_0$, so $\mathfrak h=\{0\}$. Hence the slice momentum map $J_W:W\to\mathfrak h ^*=\{0\}$ vanishes, and the momentum map~\eqref{eqn: MGS momentum map} reduces to
\begin{equation}\label{eqn: reduced MGS momentum map}
(J_K\circ \Psi^{-1})([k,\eta,u])=\Ad^*_{k^{-1}}(\nu+\eta).
\end{equation}
Since $\mathfrak h=\{0\}$, $\mathfrak m=\mathfrak k$ and thus $\mathfrak m^*=\mathfrak k^*$. Therefore the quotient map
\[
\pi_{K_{x_0}}:K\times(\mathfrak k^*\times W)\to K\times_{K_{x_0}}(\mathfrak m^*\times W)=Y
\]
is a finite covering. Shrinking $U$ if necessary, we can choose a local lift $\widehat\Psi$ of $\Psi$ through $\pi_{K_{x_0}}$, i.e.\ a map $\widehat\Psi$ making the following diagram commute:
\begin{equation}\label{diag:MGS-cover}
\begin{tikzcd}[column sep=large,row sep=large]
U \subset T^*Q
  \arrow[r,"\Psi"]
  \arrow[dr,"\widehat\Psi"']
&
U_Y \subset K\times_{K_{x_0}}(\mathfrak m^*\times W)
  \arrow[d, phantom, "{\scalebox{1.4}{$\circlearrowleft$}}"{description, xshift=-38pt, yshift=6pt}]
\\
&
U_K\times U_{\mathfrak k^*}\times U_W \subset K\times(\mathfrak k^*\times W)
  \arrow[u,"\pi_{K_{x_0}}"']
\end{tikzcd}
\end{equation}
We write
\[
\widehat\Psi(x)=(\theta(x),I(x),u(x))\in U_K\times U_{\mathfrak k^*}\times U_W.
\]
In these lifted coordinates the $K$-action is
\begin{equation}\label{eqn: K-action lifted}
k\cdot(\theta,I,u)=(k\cdot \theta,I,u)
\end{equation}
and the momentum map~\eqref{eqn: reduced MGS momentum map} has the form
\begin{equation}\label{eqn: JK}
J_K(\theta,I,u)=\Ad^*_{\theta^{-1}}(\nu+I).
\end{equation}
\begin{remark}
Since $K$ is compact, it admits a faithful finite-dimensional representation $\rho$ (see~\cite{hall2015lie}), so $\mathfrak k$ may be realized as a matrix Lie algebra. In that realization,
\[\Ad_\theta\xi=\rho(\theta)\xi\rho(\theta)^{-1},\]
and therefore equation~\eqref{eqn: JK} reads
\[J_K(\theta,I,u)(\xi)=(\nu+I)\left(\rho(\theta)\,\xi\,\rho(\theta)^{-1}\right), \qquad \forall\,\xi\in\mathfrak k.\]
If $K$ is already a matrix Lie group, this reduces to
\[J_K(\theta,I,u)(\xi)=(\nu+I)(\theta\xi\theta^{-1}).\]
\end{remark}
Define the (lifted) symplectic slice through $x_0$ by
\begin{equation}\label{eqn: symplectic slice N}
\mathfrak N := \widehat\Psi^{-1}\left(\{e\}\times\{0\}\times U_W\right)\subset U\subset T^*Q,
\end{equation}
where $e$ is the identity for $\theta$. Restricting $\widehat\Psi$ to $\mathfrak N$ gives a local symplectomorphism
\begin{equation}\label{eqn: local symplectomorphism}
 \psi:U_W\to \mathfrak N,\qquad \psi(u):=\left(\widehat \Psi|_{\mathfrak N}\right)^{-1}(e,0,u).
\end{equation}
After shrinking $U_W$ if necessary, we may assume $\pi_Q(\mathfrak N)\subset S$ (where $\pi_Q:T^*Q\to Q$ is the projection map and $S$ is the $K_{q_0}$-invariant slice through $q_0$) and that
\begin{equation}\label{eqn: sigma}
\sigma := \pi_Q\circ\psi:\ U_W\to S\subset Q
\end{equation}
is a diffeomorphism onto its image. Hence, the maps $\tilde \Phi$ and $\Phi$ in~\eqref{eqn: saturation map} and~\eqref{eqn: slice diffeomorphism} are related with the map $\psi$ in~\eqref{eqn: local symplectomorphism} by the following commutative diagram
\begin{equation}\label{diag:Phi-psi}
\begin{tikzcd}[column sep=large,row sep=large]
K\times U_W
  \arrow[r,"{\operatorname{id}\times\psi}"]
  \arrow[rr,bend left=20,"\alpha"]
  \arrow[d,"\operatorname{id}\times\sigma"']
&
K\times \mathfrak N
  \arrow[r,"\rho_K"]
  \arrow[d, phantom, "{\scalebox{1.4}{$\circlearrowleft$}}" description]
&
T^*Q
  \arrow[d,"\pi_Q"]
\\
K\times S
  \arrow[r,"\pi_{K_{q_0}}"]
  \arrow[rr,bend left=-20,"\tilde \Phi"]
&
K\times_{K_{q_0}}S
  \arrow[r,"\Phi"']
&
U_Q\subset Q
\end{tikzcd}
\end{equation}
Here $\alpha(k,u):=k\cdot\psi(u)$, the mapping $\pi_{K_{q_0}}$ is the quotient map and $\rho_K(k,x)=k\cdot x$ denotes the cotangent lifted $K$-action.

Since $K\subset G_{\xi_0}$, the augmented Hamiltonian $H_{\xi_0}$ in~\eqref{eqn:Hchi} is $K$-invariant. It follows from~\eqref{eqn: K-action lifted} that $H_{\xi_0}$ is independent of $\theta$. Therefore there exists a smooth function $\widehat H_{\xi_0}:U_{\mathfrak k^*}\times U_W\to\mathds R$ such that
\begin{equation}\label{eqn: H=widehatH}
  H_{\xi_0}(\theta,I,u) = \widehat H_{\xi_0}(I,u).
\end{equation}

\paragraph{Step 4: Stationary locked inertia tensor and drift-free condition.}

Recall from Step~3 that $S\subset Q$ is a $K$-invariant slice through $q_0$ for the $K$-action on the configuration space $Q$, and that $\mathfrak N\subset T^*Q$ is the corresponding symplectic slice through $x_0$ for the cotangent-lifted $K$-action in \eqref{eqn: symplectic slice N}. 

By~\eqref{eqn: JK}, the momentum level $J^\nu_K = J_K^{-1}(\nu)\cap U$ is characterized by  $I = \mathrm{Ad}_\theta^*\nu - \nu$. Hence, differentiating the defining relation $J_K(\theta,I,u) = \nu$ at $x_0\equiv (e,0,0)$ yields
\begin{equation}\label{eqn: TJKnu}
T_{x_0}J^\nu_K = \{(\delta\theta,\delta I,\delta u) \in  \mathfrak k\times \mathfrak k^* \times W\colon \delta I= \mathrm{ad}^*_{\delta \theta} \nu\}.
\end{equation}
Note that the characteristic (orbit) directions on $J_K^\nu$ are generated by $\mathfrak k_\nu$ (the Lie algebra of the coadjoint isotropy group $K_\nu$ of $\nu)$. If $\delta\theta\in\mathfrak k_\nu$, then  $\ad^*_{\delta\theta}\nu=0$. Substituting into~\eqref{eqn: TJKnu} yields
\begin{equation}\label{eqn:orbit-deltaI-zero}
\delta I=0 \qquad \text{for all } \delta\theta\in\mathfrak k_\nu.
\end{equation}
Moreover, by construction the $\mathfrak N$ in~\eqref{eqn: symplectic slice N} corresponds to $\mathfrak N=\{\theta = e, I=0\}$ in the lifted chart, hence $\mathfrak N \subset J_K^\nu$ and $T_{x_0} \mathfrak N\subset T_{x_0}J_K^\nu$.

For fixed $\zeta,\eta\in\mathfrak k$ consider the scalar function $f_{\zeta,\eta}$ in~\eqref{eqn: scalar function}. Then, for any curve $q(t)$ in $S$ with $q(0)=q_0$ and $\delta q := \dot q(0)\in T_{q_0}S$, the stationary condition \eqref{eqn: stationary condition} reads
\[
  0 = \frac{\mathrm d}{\mathrm dt}\bigg|_{t=0} f_{\zeta,\eta}(q(t))
    = df_{\zeta,\eta}(q_0)\cdot \delta q.
\]
Consider the map $\mathds I:Q\to \mathrm{Hom}(\mathfrak k,\mathfrak k^*)$ defined in~\eqref{eqn: locked inertia tensor}. Its differential satisfies
\[
  df_{\zeta,\eta}(q_0)\cdot \delta q
  = \big\langle (d\mathds I(q_0)\cdot \delta q)\,\zeta,\eta\big\rangle,
\]
and therefore
\begin{equation}\label{eqn:dI-vanish}
  \big\langle (d\mathds I(q_0)\cdot \delta q)\,\zeta,\eta\big\rangle = 0,\quad \text{for all } \zeta,\eta\in\mathfrak k;\; \delta q\in T_{q_0}S.
\end{equation}
By~\eqref{eqn: sigma}, for any $\delta u\in T_0U_W\simeq W$ and $\delta x := d\psi|_{0}(\delta u)\in T_{x_0}\mathfrak N$ we have
\begin{equation}\label{eqn: deltaq- deltasigma}
\delta q = d\pi_Q|_{x_0}(\delta x)= d\sigma|_{0}(\delta u)\in T_{q_0}S.
\end{equation}
Hence~\eqref{eqn:dI-vanish} applies to all directions $\delta x\in T_{x_0}N$ (equivalently, to all $\delta u\in T_0U_W$).

From \eqref{eqn: H=widehatH}, the function $H_{\xi_0}$ in \eqref{eqn:Hchi} can be realized as a smooth function $\widehat H_{\xi_0}(I,u)$ near $(I,u)=(0,0)$. Fix $(\delta I,\delta u)\in\mathfrak k^*\times W$ and denote by
\begin{equation}\label{eqn: param Hamiltonian}
F(s,t):=\widehat H_{\xi_0}(s\,\delta I,t\,\delta u).
\end{equation}
Let $q(t):=\sigma(t\,\delta u)$, which is well-defined for $|t|$ small. Then $q(0)=q_0$ and $\dot q(0)=d\sigma|_0(\delta u)$, so~\eqref{eqn: deltaq- deltasigma} holds. We denote by 
\begin{equation}\label{eqn: terms depending on t}
\begin{aligned}
A(s,t)&:=J(s,t)-\mathbb I(q(t))\xi_0, \qquad G(s,t):=\tfrac12\left\langle A(s,t),\mathbb I(q(t))^{-1}A(s,t)\right\rangle\\
R(t)&:=-\frac12\langle \mathbb I(q(t))\xi_0,\xi_0\rangle +\frac12 g_{q(t)}^{-1}(p_H(t),p_H(t))+V(q(t))
\end{aligned}
\end{equation}
so that $F(s,t)=G(s,t)+R(t)$. Differentiating $G$ in the $s$- and $t$-directions, and applying~\eqref{eqn: locked inertia tensor} and~\eqref{eqn: xi0}, we obtain 
\[
\partial_t\partial_s G(0,0)
=\left\langle \partial_s A(0,0),\,\mathbb I(q_0)^{-1}\,\partial_t A(0,0)\right\rangle.
\]
On the one hand, $J_K$ in~\eqref{eqn: MGS momentum map} is independent of $u$, so we have $\partial_t J_K(\theta(s),I(s),u(t)) = 0$. In particular, on $\mathfrak N$ one has $J_K(e,0,u)=\nu$, hence $A(0,t) = \nu - \mathbb I(q(t))\xi_0$ and therefore 
\[\partial_t A(0,0)=-\left(d\mathbb I(q_0)\cdot \delta q\right)\xi_0.\]
On the other hand, $\partial_s A(0,0)=\partial_s J(0,0)$ lies in $\mathfrak k^*$ and $\mathbb I(q_0)^{-1}\partial_sA(0,0)\in \mathfrak k$. Therefore $\partial_t\partial_s G(0,0)$ is of the form
\[
-\left\langle \left(d\mathbb I(q_0)\cdot \delta q\right)\zeta,\eta\right\rangle
\quad \text{for some }\zeta,\eta\in\mathfrak k,
\]
and hence it vanishes by~\eqref{eqn:dI-vanish}, since $\delta q\in T_{q_0}S$ by \eqref{eqn: deltaq- deltasigma}. The function $R(t)$ in~\eqref{eqn: terms depending on t} depends only on $t$, so the mixed derivative $\partial_t\partial_s R(t)$ also vanishes and consequently
\[
\partial_I\partial_u\widehat H_{\xi_0}(0,0)= \partial_t\partial_s F(0,0) = 0.
\]
Since $x_0$ is a critical point of $H_{\xi_0}$ in~\eqref{eqn:Hchi}, we also have $\partial_I \widehat H_{\xi_0}(0,0) = 0$. Therefore the quadratic part $H_{2,\xi_0}$ of $H_{\xi_0}$ satisfies
\[
  \partial_I H_{2,\xi_0}(0,u) = 0 \quad\text{for all }u.
\]
The restriction of the quadratic term $H_{2,\xi_0}$ to the momentum level $J_K^\nu$ depends only on the tangent relation~\eqref{eqn: TJKnu}, and along the characteristic (orbit) directions $\delta \theta \in \mathfrak k_\nu$ one has $\delta I = 0$ (see~\eqref{eqn:orbit-deltaI-zero}). Hence the linear Hamiltonian vector field associated to $H_{2,\xi_0}|_{J_K^\nu}$ satisfies~\eqref{eqn: drift-free derivative condition}, and therefore it is drift-free near $x_0$ on $J_K^\nu$ (see Definition~\ref{def:drift-free}). This completes the proof.

\section{The double spherical pendulum}\label{sec: double spherical pendulum}
The double spherical pendulum is a classical Hamiltonian system with a nontrivial symmetry group. It has served as a standard example for singular reduction and local bifurcation/normal form analysis in a number of works, such as~\cite{CIFTCI20141, marsden1993lagrangian}. In these references, the normal form computations are carried out on a reduced phase space, obtained either by Lagrangian reduction or by explicit symplectic reduction on the cotangent bundle. This requires introducing local canonical coordinates on a neighborhood of the orbit in the reduced space, which is typically a stratified, globally singular symplectic space.

In this section we focus on a specific relative equilibrium of the double spherical pendulum: the \emph{spinning configuration}, in which the whole configuration spins steadily about the vertical axis with angular velocity~$\Omega$ (see Figure~\ref{fig:double_sphere_equilibria}). All constructions below are local in phase space around this relative equilibrium and take place on a suitable nonzero momentum level.

We show how to carry out the Birkhoff normal form analysis for such a relative equilibrium directly on the unreduced phase space, which is a smooth symplectic manifold. To this end, first we check that the hypotheses of Theorem~\ref{thm:drift-free-mech} are satisfied, so that the quadratic part of the associated augmented Hamiltonian is drift-free near the relative equilibrium on the relevant momentum level set for the $S^1$-action. Consequently, Theorem~\ref{thm:main-BNF} applies: we work on a smooth constraint submanifold equipped with the Dirac bracket adapted to the momentum constraint, and the resulting Birkhoff normal form agrees with the Poincaré-Birkhoff normal form on the symplectic stratum of the (possibly singular) reduced space.

We proceed as follows. First we recall the basic setup and symmetry of the double spherical pendulum, and we consider the spinning relative equilibrium. Then we identify the double spherical pendulum as a simple mechanical system with symmetry (see Definition~\ref{def:simple-mech-stationary}) and we verify the hypotheses of Theorem~\ref{thm:drift-free-mech}. Finally, we construct the Birkhoff normal form on the momentum level using the Dirac bracket.


\paragraph{Setup and Hamiltonian.}
We consider two point masses $m_1,m_2>0$ attached by massless rigid rods of lengths $l_1,l_2>0$. The first rod is attached to a fixed pivot at the origin, and the second rod is attached to the first mass. Let $q_1,q_2\in S^2\subset\mathds R^3$ be the unit vectors along each rod, pointing from the pivot to the mass. The positions of the masses are
\[
r_1 = l_1 q_1,\qquad r_2 = l_1 q_1 + l_2 q_2.
\]
We choose the vertical unit vector $e_3=(0,0,1)$ pointing upwards, so that gravity acts along $-e_3$.

The configuration space is $Q=S^2\times S^2$, and the phase space is $M = T^*S^2\times T^*S^2$, equipped with the canonical symplectic form $\omega = \omega_{\mathrm{can}}\oplus\omega_{\mathrm{can}}$. The velocities are
\[
\dot r_1 = l_1 \dot q_1, \qquad \dot r_2 = l_1 \dot q_1 + l_2 \dot q_2.
\]
The kinetic energy and the potential energy are given by
\begin{equation}\label{eqn: kinetic and potential energy}
\begin{aligned}
T(\dot q_1,\dot q_2) &= \frac12 m_1 \|\dot r_1\|^2 + \frac12 m_2 \|\dot r_2\|^2 = \frac12 (m_1+m_2) l_1^2 \|\dot q_1\|^2 + \frac12 m_2 l_2^2 \|\dot q_2\|^2 + m_2 l_1 l_2\,\dot q_1\cdot\dot q_2,\\
V(q_1,q_2) &= (m_1+m_2)g l_1 (1 + q_1\cdot e_3) + m_2 g l_2 (1 + q_2\cdot e_3).
\end{aligned}
\end{equation}
The Lagrangian is $L=T-V$. The canonical momenta $p_i\in T_{q_i}^*S^2$ are
\[
p_1 = \frac{\partial L}{\partial \dot q_1} = (m_1+m_2)l_1^2\dot q_1 + m_2 l_1 l_2 \dot q_2,\qquad p_2 = \frac{\partial L}{\partial \dot q_2} = m_2 l_2^2 \dot q_2 + m_2 l_1 l_2 \dot q_1.
\]
This is a linear relation between $(\dot q_1,\dot q_2)$ and $(p_1,p_2)$. The corresponding $4\times 4$ ``mass matrix'' has determinant $\Delta = m_1 m_2 l_1^2 l_2^2>0$, so it is invertible. Therefore we have
\begin{equation*}
\dot q_1 = \frac{m_2 l_2^2}{\Delta}\,p_1 - \frac{m_2 l_1 l_2}{\Delta}\,p_2,\qquad \dot q_2 = - \frac{m_2 l_1 l_2}{\Delta}\,p_1 + \frac{(m_1+m_2) l_1^2}{\Delta}\,p_2.
\end{equation*}
The Hamiltonian $H = p_1\cdot\dot q_1+p_2\cdot\dot q_2 - L$ then takes the form
\begin{equation}\label{eqn:double-H}
H(q_1,q_2,p_1,p_2)= \frac12\left(\alpha_{11}\|p_1\|^2 + 2\alpha_{12}\,p_1\cdot p_2 + \alpha_{22}\|p_2\|^2\right)+ (m_1+m_2)g l_1\,q_1\cdot e_3 + m_2 g l_2\,q_2\cdot e_3,
\end{equation}
where
\begin{equation}\label{eqn: alphacte}
\alpha_{11} = \frac{m_2 l_2^2}{\Delta},\qquad \alpha_{22} = \frac{(m_1+m_2) l_1^2}{\Delta},\qquad \alpha_{12} = -\frac{m_2 l_1 l_2}{\Delta}.
\end{equation}
Explicit local coordinate expressions equivalent to~\eqref{eqn:double-H} can be found in~\cite{marsden1993lagrangian}.

\paragraph{Momentum map and relative equilibria.}
Let $G=S^1$ act on $Q=S^2\times S^2$ diagonally by rotations around the vertical axis,
\[
R_\alpha\cdot(q_1,q_2) = (R_\alpha q_1, R_\alpha q_2),\qquad \alpha\in S^1,
\]
where $R_\alpha\in \mathrm{SO}(3)$ is the rotation through the angle $\alpha$ about $e_3$. This action lifts to a Hamiltonian action on $M=T^*S^2\times T^*S^2$ by cotangent lifts, preserving the symplectic form and the Hamiltonian~\eqref{eqn:double-H}. The corresponding momentum map $J:M\to\mathfrak g^*\simeq\mathds R$ is the vertical component of the total angular momentum,
\[
J(q_1,q_2,p_1,p_2) = e_3\cdot\left( r_1\times p_{r_1} + r_2\times p_{r_2} \right),
\]
where $p_{r_i}$ denotes the generalized momentum conjugate to $r_i$ \footnote{Writing $J$ in terms of $(q_i,p_i)$ is straightforward, but not particularly useful in this case.}. The geometric expression above is standard in the literature on the double spherical pendulum (see e.g.~\cite{CIFTCI20141,marsden1993lagrangian}). The Hamiltonian $H$ in~\eqref{eqn:double-H} is $G$-invariant, so $\{H,J\}_M=0$. For each $\mu\in\mathds R$ we consider the momentum level set
\begin{equation}\label{eqn: momentum level set example}
J_\mu := J^{-1}(\mu)\subset M.
\end{equation}
We look for a relative equilibrium in which the entire configuration spins steadily around the vertical axis with angular velocity $\Omega\in\mathds R$. In the spatial frame this corresponds to a periodic motion (see Figure \ref{fig:double_sphere_equilibria}). In a frame rotating with angular velocity $\Omega e_3$ the configuration is stationary. Passing to such rotating frame amounts to considering the \emph{augmented} Hamiltonian
\begin{equation}\label{eqn: HOmega}
  H_\Omega := H - \Omega J.
\end{equation}
A point $x_0=(q_1^0,q_2^0,p_1^0,p_2^0)\in M$ is a relative equilibrium of the original system with angular velocity $\Omega$ if and only if it is an equilibrium of $H_\Omega$, that is, if $dH_\Omega(x_0)=0$.

One finds that, for suitable values of $\Omega$ and $\mu\neq 0$, there exist momenta $(p_1^0,p_2^0)$ tangent to $S^2\times S^2$ at $(q_1^0,q_2^0)$ such that $x_0$ is a critical point of $H_\Omega$ with $J(x_0)=\mu$ (for explicit algebraic conditions and the resulting stability and bifurcation analysis we refer to \cite[Sec.~4-5]{marsden1993lagrangian}). For such a choice, we denote by
\begin{equation}\label{eqn: spinning relative equilibrium}
x_{\mathrm{Sp}}\in J_\mu,\quad \mu \neq 0
\end{equation}
the \emph{spinning} relative equilibrium on a nonzero momentum level. Its image $[x_{\mathrm{Sp}}]\in M_\mu := J_\mu/S^1$ (given by the quotient map $\pi_{\mu}$ in \eqref{eqn: quotient map}) lies in a $6$-dimensional symplectic stratum $\mathcal S\subset M_\mu$ (corresponding to trivial isotropy). In this case, the induced $S^1$-action is locally free in a neighborhood of $x_{\mathrm{Sp}}$ (see~\cite{CIFTCI20141,marsden1993lagrangian} for a detailed description of $M_\mu$ and its topology). 

The case $\mu=0$ is \emph{singular} (see Section~\ref{subsec:limiting}). The set $J_0=J^{-1}(0)$ contains points with nontrivial isotropy, as well as points with trivial isotropy. In fact, for the diagonal $S^1$-action by rotations about the vertical axis, the only points with full $S^1$-isotropy in $J_0$ are the vertical configurations in which each rod is aligned with $\pm e_3$ and the corresponding momenta vanish. There are exactly four such fixed points. Therefore the quotient $M_0:=J_0/S^1$ is a stratified symplectic space consisting of a $6$-dimensional open symplectic stratum (corresponding to trivial isotropy) together with four $0$-dimensional singular strata (isolated points) corresponding to these fixed configurations. The hypotheses of Theorem~\ref{thm:general_intertwining_JS} fail at these configurations (and consequently Theorem~\ref{thm:main-BNF} cannot be applied in this case). These singularities are analyzed in~\cite{marsden1993lagrangian}, where blow-up techniques are used to study the local dynamics.

\begin{remark}
The blow-up techniques in~\cite{marsden1993lagrangian} are used to regularize the linearized equations and their characteristic polynomial near the singular limit. However, as the authors put it in~\cite[Sec.~1, Page~2]{marsden1993lagrangian}: 
\begin{quote}
``Our suggested approach to this problem, which is only sketched, is that of blowing up singularities and regularization. However, no attempt is made to give a complete account, or to relate our method to that of singular reduction... This would be of considerable interest, but is not the purpose of the present paper to explore.''
\end{quote}
Thus, while a blow-up may provide a useful tool for analyzing the local dynamics in this singular regime, extending our Dirac-slice framework to the blown-up space would require additional work and lies beyond the scope of the present paper.
\end{remark}

\begin{figure}[ht]
    \centering
    \includegraphics[width=\linewidth]{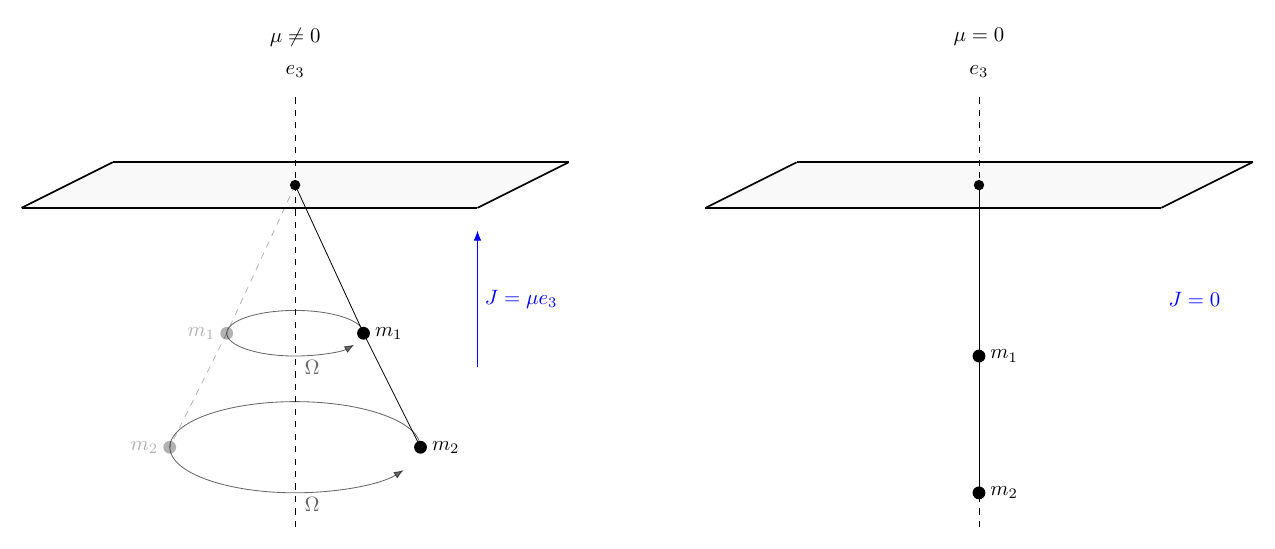}
    \caption{Double spherical pendulum in two different configurations. On the left, the spinning relative equilibrium on a nonzero momentum level ($\mu \ne 0$), with constant angular velocity $\Omega$ and vertical angular momentum $J=\mu e_ 3$. On the right, the static equilibrium  with zero angular momentum ($\mu = 0$), whose image in the reduced space $M_0$ is a singular point corresponding to full $S^1$ isotropy.}
    \label{fig:double_sphere_equilibria}
\end{figure}

\paragraph{Simple mechanical system with symmetry and stationary locked inertia tensor.}
Let $g$ be the Riemannian metric determined by the kinetic energy $T$ and $V$ be the potential, both defined in~\eqref{eqn: kinetic and potential energy}. Then $(Q,g,V,G)=(S^2\times S^2,g,V,S^1)$ is a simple mechanical system with symmetry with Hamiltonian~\eqref{eqn:double-H}. Since $\mathfrak g\simeq\mathds R$, the locked inertia tensor in~\eqref{eqn: locked inertia tensor} reduces to a scalar function $\mathds I(q_1,q_2)\in\mathds R$, characterized by~\eqref{eqn: scalar function} as
\[
  \langle \mathds I(q_1,q_2)\,\zeta,\zeta\rangle
  = g_{(q_1,q_2)}\left(\zeta_Q(q_1,q_2),\zeta_Q(q_1,q_2)\right),
  \qquad \zeta\in\mathds R,
\]
where
\[
  \zeta_Q(q_1,q_2)=\left(\zeta\,e_3\times q_1,\ \zeta\,e_3\times q_2\right)
\]
is the infinitesimal generator of the $S^1$-action. Using~\eqref{eqn: kinetic energy locked inertia} and~\eqref{eqn: kinetic and potential energy} we obtain
\begin{equation}\label{eqn:locked-inertia-double-pendulum}
  \mathds I(q_1,q_2)
  = (m_1+m_2)l_1^2\|e_3\times q_1\|^2
    + m_2l_2^2\|e_3\times q_2\|^2
    + 2m_2l_1l_2 (e_3\times q_1)\cdot(e_3\times q_2).
\end{equation}
Since $\mathds I$ is $S^1$-invariant, at points with trivial isotropy condition~\eqref{eqn: stationary condition} is equivalent to $q_{\mathrm{Sp}}$ being a critical point of the scalar function $\mathds I$ on $Q$, i.e.\ $d\mathds I(q_{\mathrm{Sp}})=0$. Setting
\[
A:=(m_1+m_2)l_1^2,\qquad B:=m_2l_2^2,\qquad C:=2m_2l_1l_2,
\]
the expression~\eqref{eqn:locked-inertia-double-pendulum} is written as
\[
\mathds I(q_1,q_2)=A\|e_3\times q_1\|^2+B\|e_3\times q_2\|^2+C\,(e_3\times q_1)\cdot(e_3\times q_2),
\qquad (q_1,q_2)\in S^2\times S^2.
\]
Let $(\delta q_1,\delta q_2)\in T_{(q_1,q_2)}(S^2\times S^2)$, so $q_i\cdot\delta q_i=0$ for $i=1,2$. Writing $z_i:=e_3\cdot q_i$, we have $\|e_3\times q_i\|^2=1-z_i^2$. Hence
\[
\mathrm d\left(\|e_3\times q_i\|^2\right)\cdot\delta q_i=-2z_i\,(e_3\cdot\delta q_i).
\]
Moreover, using $(e_3\times q_1)\cdot(e_3\times q_2)=q_1\cdot q_2-(e_3\cdot q_1)(e_3\cdot q_2)$ we obtain
\[
\mathrm d\!\left((e_3\times q_1)\cdot(e_3\times q_2)\right)\cdot(\delta q_1,\delta q_2) =\delta q_1\cdot q_2+q_1\cdot\delta q_2 -z_2(e_3\cdot\delta q_1)-z_1(e_3\cdot\delta q_2).
\]
Combining these identities yields
\begin{equation}\label{eqn:dI-first-variation}
\begin{aligned}
\mathrm d\mathds I(q_1,q_2)\cdot(\delta q_1,\delta q_2)=\left\langle\, Cq_2-(2Az_1+Cz_2)e_3,\ \delta q_1\right\rangle +\left\langle\, Cq_1-(2Bz_2+Cz_1)e_3,\ \delta q_2\right\rangle .
\end{aligned}
\end{equation}
Equivalently
\[
\nabla_{q_1}\mathds I=Cq_2-(2Az_1+Cz_2)e_3, \qquad \nabla_{q_2}\mathds I=Cq_1-(2Bz_2+Cz_1)e_3.
\]
Therefore $(q_1,q_2)$ is a critical point of $\mathds I$ on $S^2\times S^2$ if and only if there exist scalars $\lambda_1,\lambda_2\in\R$ such that
\begin{equation}\label{eqn:dI-criticality}
\nabla_{q_1}\mathds I=\lambda_1 q_1,\qquad \nabla_{q_2}\mathds I=\lambda_2 q_2,
\end{equation}
i.e.\ the tangential components of $\nabla_{q_i}\mathds I$ vanish. Solving~\eqref{eqn:dI-criticality} yields the following stationary configurations:
\begin{enumerate}
\item $q_1,q_2\in\{\pm e_3\}$ (the \emph{static} equilibrium with zero momentum level);
\item $q_1\cdot e_3=0$ and $q_2\cdot e_3=0$ (both rods horizontal) with $q_2=\pm q_1$;
\item $q_1\cdot e_3=0$ and $e_3\times q_1 = - \frac{2B}{C} (e_3\times q_2)$.
\item $q_2\cdot e_3=0$ and $e_3\times q_2 = -\frac{2A}{C}(e_3\times q_1)$;
\end{enumerate}
In cases (2)-(4) the isotropy is trivial and one has $\mathds I(q_1,q_2)>0$, so these configurations occur in the locally free region and are compatible with nonzero momentum levels. For such spinning equilibria $x_{\mathrm{Sp}}\in J^{-1}(\mu)$ with $\mu\neq 0$ and stationary locked inertia tensor, the hypotheses of Theorem~\ref{thm:drift-free-mech} are satisfied, and hence the quadratic part of the augmented Hamiltonian $H_\Omega$ is drift-free near $x_{\mathrm{Sp}}$ on $J_\mu$ with respect to the $S^1$-action.

\begin{figure}[ht]
    \centering
    \includegraphics[width=\linewidth]{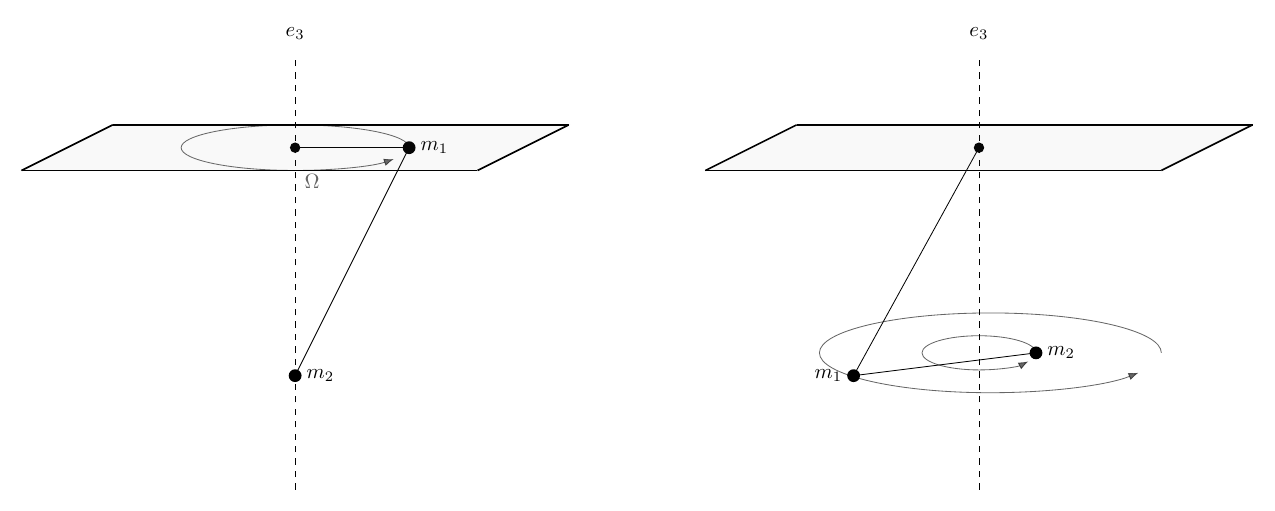}
    \caption{Representation of the spinning relative equilibria corresponding to cases $(3)$ and $(4)$, respectively. In case $(3)$, $q_1$ is horizontal; the first mass rotates about $e_3$, whereas the second one lines on the $e_3$-axis and is fixed in space. In case $(4)$, $q_2$ is horizontal and both masses rotate about $e_3$. Both configurations have trivial isotropy and occur on nonzero momentum levels.}
    \label{fig: admissible configurations}
\end{figure}
\begin{remark}
In cases (3) and (4) (see Figure~\ref{fig: admissible configurations}) the stated relations are compatible with $q_i\in S^2$  only if the corresponding norms are not bigger than $1$. Namely
\begin{equation*}
    \begin{aligned}
        \frac{C}{2B} &= \frac{l_1}{l_2}\leq 1,\quad &\text{in (3)}\\
        \frac{C}{2A} &= \frac{m_2}{m_1+m_2}\cdot \frac{l_2}{l_1}\leq 1,\quad &\text{ in (4)}.
    \end{aligned}
\end{equation*}
\end{remark}

\paragraph{Second-class constraints and Dirac bracket.}
We apply the general second-class constraint construction given in Section~\ref{sec: Proof of Thms} to the double spherical pendulum on the fixed momentum level $J_\mu$ in~\eqref{eqn: momentum level set example}. Let $Z\subset T^*Q$ be a neighborhood of $x_{\mathrm{Sp}}$ defined in \eqref{eqn: spinning relative equilibrium} and denote by
\begin{equation}\label{eqn: set U}
U := J_\mu \cap Z.
\end{equation}
Since $G=S^1$ and $x_{\mathrm{Sp}}$ has trivial isotropy, we have $\ell=1$ in~\eqref{eqn: dim l}. Choosing the standard generator $\xi\in\mathfrak{so}(2)\simeq\mathds R$ and the identification $\mathfrak g^*\simeq\mathds R$ given by $\alpha\mapsto \langle \alpha,\xi\rangle$, we view $J_\xi(x):=\langle J(x),\xi\rangle$ and $\mu_\xi:=\langle\mu,\xi\rangle$ as real numbers. Then the momentum-type constraint in~\eqref{eqn: Phi momentum map} reduces to $\Phi(x)=\langle J(x)-\mu,\xi\rangle = J(x)-\mu$. Choose a smooth hypersurface $\Sigma$ transverse to the $S^1$-orbits through points of $Z$ near $x_{\mathrm{Sp}}$ as in~\eqref{eqn: Sigma}, given locally by $\Sigma \cap Z = \{\Upsilon=0\}$ (see~\eqref{eqn: chi=0}). Then
\begin{equation}\label{eqn: N double spherical pendulum}
  N := U \cap \Sigma = \{x \in U \colon J(x)=\mu,\ \Upsilon(x)=0\}
\end{equation}
is a $6$-dimensional local symplectic slice through $x_{\mathrm{Sp}}$ (see~\eqref{eqn: N}), and $(J-\mu,\Upsilon)$ is a local system of second-class constraints defining $N$ (see Figure~\ref{fig:second_class_constraint}).

Since the quadratic part of $H_\Omega$ is drift-free near $x_{\mathrm{Sp}}$ on $J_\mu$ with the $S^1$-action, Theorems~\ref{thm:general_intertwining_JS} and~\ref{thm:main-BNF} yield that, on $U$, the Hamiltonian vector fields generated by $H_\Omega$ with respect to the Dirac bracket associated with $(J-\mu,\Upsilon)$ coincide with those generated by the induced bracket $\{\cdot,\cdot\}_{J_\mu}$. In particular, the Birkhoff normal form constructed on $J_\mu$ with this bracket agrees with the Birkhoff normal form on the symplectic stratum $S \subset M_\mu$ containing $[x_{\mathrm{Sp}}]$.

\paragraph{Birkhoff normal form on the momentum level.}
We introduce linear Darboux coordinates on the symplectic vector space $(T_{x_{\mathrm{Sp}}}N,\omega_N(x_{\mathrm{Sp}}))$, where $\omega_N:=\iota_N^*\omega$. Choose local coordinates $\hat q$ on $Q=S^2\times S^2$ centered at $(q_1^0,q_2^0)$ and let $(\hat q,\hat p)$ be the induced canonical coordinates on $T^*Q$ near $x_{\mathrm{Sp}}$. After shrinking if necessary, since $d(J-\mu)$ and $d\Upsilon$ are linearly independent along $N$ (they are second-class constraints), the constant rank theorem yields local coordinates
\[
(J-\mu,\Upsilon,z)\colon Z\to \mathds R^2\times \mathds R^6,
\]
so that $z$ restricts to local coordinates on $N=\{J=\mu,\Upsilon=0\}$ in~\eqref{eqn: N double spherical pendulum} and we may thus identify $T_{x_{\mathrm{Sp}}}N\simeq\mathds R^6$. Finally, we apply a linear symplectic change of variables to obtain Darboux coordinates $z=(Q_1,Q_2,Q_3,P_1,P_2,P_3)$ on $T_{x_{\mathrm{Sp}}}N$. Via the local diffeomorphism $\pi_{\mu}|_N\colon N\to V\subset \mathcal S$ induced by the quotient map in~\eqref{eqn: quotient map}, these also define linear Darboux coordinates on $T_{[x_{\mathrm{Sp}}]}\mathcal S$. In these coordinates, the Taylor expansion of $H_\Omega$ in~\eqref{eqn: HOmega} at $x_{\mathrm{Sp}}$ reads
\[
H_\Omega(z) = E_{\mathrm{Sp}} + H_{2,\Omega}(z) + H_3(z) + H_4(z) + \cdots,
\]
where $E_{\mathrm{Sp}} = H_\Omega(x_{\mathrm{Sp}})$ is the energy of the relative equilibrium, $H_{2,\Omega}$ is the quadratic part, and $H_k$ is homogeneous of degree $k$ in $z$.

Using the explicit expressions for the linearized dynamics given in~\cite{CIFTCI20141,marsden1993lagrangian}, one finds that $H_{2,\Omega}$ can be brought, by a linear symplectic transformation, into
\[
H_{2,\Omega}=\frac12\sum_{j=1}^3 \eta_j(P_j^2+Q_j^2)
\]
in suitable canonical coordinates $(Q_1,Q_2,Q_3,P_1,P_2,P_3)$, where the $\eta_j$ are the squared frequencies (up to signs, depending on the stability type).

We now construct the Birkhoff normal form of $H_\Omega$ near $x_{\mathrm{Sp}}$ on $U\subset J_\mu$ in~\eqref{eqn: set U} using the induced bracket $\{\cdot,\cdot\}_{J_\mu}$. The normal form computation is carried out in $N$, where $\{\cdot,\cdot\}_{J_\mu}$ is nondegenerate and a set of Darboux coordinates is available. As we work in the pullback algebra $\pi_{\mu}^*C^\infty(V)$, the resulting transformations and the normal form Hamiltonian extend from $N$ to $U$ by $S^1$-equivariance, yielding a normal form statement on $U\subset J_\mu$ (see Section~\ref{sec: Proof of Thms}). Hence, at each order $k\ge3$, we can solve the homological equation
\[
 \mathcal L_{J_\mu}K_k + H_k^{\mathrm{nr}} = 0,
\]
where $\mathcal L_{J_\mu}$ is defined in~\eqref{eqn: LJ and LS H2} and $H_k^{\mathrm{nr}}$ denotes the nonresonant part of $H_k$ with respect to the linear flow generated by $H_{2,\Omega}$. The generating function $K_k$ defines a near-identity canonical transformation (with respect to $\{\cdot,\cdot\}_{J_\mu}$) which removes $H_k^{\mathrm{nr}}$ while preserving the lower order terms. Iterating this procedure order by order yields a formal Hamiltonian
\[
  H_{\mathrm{BNF}} = E_{\mathrm{Sp}} + H_{2,\Omega} + \widehat H_3 + \widehat H_4 + \cdots,
\]
where each $\widehat H_k$ Poisson-commutes with $H_{2,\Omega}$ with respect to $\{\cdot,\cdot\}_{J_\mu}$. 
\begin{remark}
    The explicit fourth-order reduced Birkhoff normal form obtained in~\cite[Sec.~5.3]{CIFTCI20141} is computed for a different relative equilibrium, for which the stationarity criterion for the locked inertia tensor is not satisfied. Although the drift-free condition may still hold in that case, proving it would require a separate analysis and goes beyond the scope of the present paper. 
\end{remark}

\section{Extension to Moser's constrained systems and near-integrability}
\label{sec:Moser-extension}

We consider near-integrable systems in the following sense:

\begin{definition}
\label{def: near-integrable}
Let $(M,\omega)$ be a symplectic manifold of dimension $2n$ and let $H_\varepsilon=H_0+\varepsilon H_1$ be a family of Hamiltonians, where $\varepsilon\in(-\varepsilon_0,\varepsilon_0)$ for some $\varepsilon_0>0$. We say that $H_\varepsilon$ is \emph{near-integrable on an open set $V\subset M$} if $H_0$ is Liouville integrable on $V$ and, after possibly restricting to a smaller open set, there exist action-angle coordinates $(\theta,I)\in \mathds T^n\times D$ on $V$ such that
\[
  H_\varepsilon(\theta,I)=H_0(I)+\varepsilon H_1(\theta,I;\varepsilon).
\]
\end{definition}

In this section, we explain how Theorem~\ref{thm:general_intertwining_JS} provides a general mechanism to address the following question motivated by Moser in~\cite[Sec.~2.5]{1971430859838502419}, who posed it for integrable systems.

\begin{conjecture}\label{conj: near integrable}
If a Hamiltonian system is a small perturbation of an integrable one, can we impose constraints (or perform reduction) in such a way that the constrained (or reduced) system is again near-integrable in the same sense?
\end{conjecture}

In Section~\ref{subsec:Moser-constrained} we recall Moser's constrained vector field and his integrable construction. In Section~\ref{subsec:Moser-Dirac} we reinterpret them as instances of the Dirac bracket formalism. Finally, in Section~\ref{subsec:stratum-near-integrability} we use Theorem~\ref{thm:general_intertwining_JS} to give a general positive answer to Question~\ref{conj: near integrable}. 

\subsection{Moser's constrained Hamiltonian systems}\label{subsec:Moser-constrained}

Let $\mathds R^{2n}$ be endowed with the canonical symplectic form $\omega_0$ and the standard Poisson bracket $\{\cdot,\cdot\}$. Moser considered a smooth submanifold $N\subset \mathds R^{2n}$ defined by $2r$  independent constraint functions
\[
  G_1=\cdots=G_{2r}=0,\quad \text{for some } r<n,
\]
and assumes that the constraints are \emph{nondegenerate} in the sense that the $2r\times2r$ matrix
\begin{equation}
  C(z) := \left(\{G_i,G_j\}(z)\right)_{i,j=1}^{2r}
  \label{eqn:Moser-C}
\end{equation}
is invertible at all $z\in N$. Under this hypothesis, $N=\{G_1=\dots=G_{2r}=0\}$ is a symplectic submanifold of $(\mathds R^{2n},\omega_0)$ and inherits the restriction of the ambient symplectic form. 

Given a Hamiltonian $H\in C^\infty(\mathds R^{2n})$, the ambient Hamiltonian vector field $X_H$ is in general not tangent to $N$. Moser looked for a vector field $X_H^*$ on $\mathds R^{2n}$ obtained by adding a linear combination of the constraint vector fields:
\begin{equation}\label{eqn:Moser-X-star}
  X_H^* \colon= X_H - \sum_{s=1}^{2r}\lambda_s X_{G_s}, 
\end{equation}
where the multipliers $\lambda_s=\lambda_s(z)$ are determined by the requirement that $X_H^*$ be tangent to $N$. Equivalently,
\[
  \mathcal L_{X_H^*} G_k = 0\quad\text{on }N,\qquad k=1,\dots,2r,
\]
that is,
\[
  \{H,G_k\} - \sum_{s=1}^{2r} \lambda_s \{G_s,G_k\} = 0 \quad\text{on }N,\quad k=1,\dots,2r.
\]
By the nondegeneracy of $C(z)$ in~\eqref{eqn:Moser-C}, this linear system uniquely determines the multipliers $\lambda_s(z)$ for each $z\in N$. Writing
\[
  H^* := H - \sum_{s=1}^{2r}\lambda_s G_s,
\]
the constrained vector field on $N$ can be expressed as $X_H^*$ in~\eqref{eqn:Moser-X-star}. By construction, $X_H^*$ is tangent to $N$, and Moser observes that it does not depend on how the multipliers $\lambda_s$ are extended off $N$.

Suppose that $X_H$ admits $n$ functionally independent, commuting integrals $F_1,\dots,F_n$, so that
\[
  \{F_i,F_j\} = 0,\quad i,j=1,\dots,n.
\]
Moser introduced a special class of constraint manifolds of the form
\begin{equation}\label{eqn:Moser-constraint-manifold}
  N \;=\; \{\,G_1=\dots=G_r=0,\; F_1=\dots=F_r=0\,\},
\end{equation}
where $G_1,\dots,G_r$ are additional functions, and assumes that the $2r$ functions $G_1,\dots,G_r,F_1,\dots,F_r$ satisfy the canonical relations 
\begin{equation}\label{eqn:Moser-canonical-constraints}
  \{G_i,G_j\}=0,\quad \{F_i,F_j\}=0,\quad \{G_i,F_j\} = \delta_{ij},\qquad i,j=1,\dots,r.
\end{equation}
In Moser's notation this means that, in~\eqref{eqn:Moser-C}, one has chosen $G_{r+j}=F_j$ for $j=1,\dots,r$. Under these hypotheses he proves:

\begin{proposition}\label{prop:Moser-integrable}
Let $X_H$ be integrable on $\mathds R^{2n}$ with commuting integrals $F_1,\dots,F_n$. Assume that $N$ is given by~\eqref{eqn:Moser-constraint-manifold} and that the functions $G_1,\dots,G_r,F_1,\dots,F_r$ satisfy~\eqref{eqn:Moser-canonical-constraints}. Then the constrained vector field $X_H^*$ on $N$ admits the restrictions
\[
  F_{r+1}|_N,\dots,F_n|_N
\]
as commuting first integrals, so in particular $X_H^*$ is integrable on $N$.
\end{proposition}
In particular, Moser illustrates this construction by constraining a harmonic oscillator on $\mathds R^{2n}$ to the tangent bundle of the unit sphere, obtaining the Neumann system as an integrable constrained system (see~\cite{10.1007/978-1-4613-8109-9_7}). 

\subsection{Dirac brackets and Moser's construction}
\label{subsec:Moser-Dirac}

This section is devoted to explain how Moser's nondegenerate constraints correspond to the second-class constraints from  Lemma~\ref{lemma:Dirac bracket}. 

Let $(M,\omega)$ be a symplectic manifold with Poisson bracket $\{\cdot,\cdot\}_M$, and let
\[
  N = \{x\in M : \phi_1(x)=\cdots=\phi_{2r}(x)=0\}
\]
be a regular submanifold defined by independent constraint functions $\phi_i\in C^\infty(M)$. Set $C_{ij}=\{\phi_i,\phi_j\}_M$ and assume that the $2r\times2r$ matrix $C=(C_{ij})$ is invertible along $N$, so that the constraints are second-class in the sense of Lemma~\ref{lemma:Dirac bracket}. Then the Dirac bracket along $N$ is defined by 
\begin{equation}
  \{F,G\}_D \;=\; \{F,G\}_M - \sum_{i,j=1}^{2r}
   \{F,\phi_i\}_M (C^{-1})_{ij}\, \{\phi_j,G\}_M,
   \qquad F,G\in C^\infty(M),
  \label{eqn:Dirac-bracket}
\end{equation}
and induces a nondegenerate Poisson structure on $N$. The corresponding Hamiltonian vector field of $H$ with respect to $\{\cdot,\cdot\}_D$ can be written on $N$ as
\begin{equation}\label{eqn:Dirac-vector-field}
  X_H^D = X_H - \sum_{i,j=1}^{2r} \{\phi_i,H\}_M (C^{-1})_{ij}\, X_{\phi_j}, \qquad\text{on }N.
\end{equation}
If we apply this construction to the constraints considered by Moser, we take $M=\mathds R^{2n}$, $\{\cdot,\cdot\}_M=\{\cdot,\cdot\}$ the canonical bracket, and
\[
  \phi = (G_1,\dots,G_{2r}), \qquad N = \{G_1=\cdots=G_{2r}=0\}.
\]
Then the nondegeneracy assumption~\eqref{eqn:Moser-C} says exactly that $\phi$ is a set of second-class constraints. The constrained vector field $X_H^*$ defined by Moser in~\eqref{eqn:Moser-X-star} coincides with the Dirac Hamiltonian vector field $X_H^D$ restricted to $N$ in~\eqref{eqn:Dirac-vector-field}, and the dynamics of $H$ on $N$ can be equivalently described by
\[
  \dot{F}|_N = \{F,H\}_D\big|_N, \qquad F\in C^\infty(\mathds R^{2n}).
\]
In particular, in the special integrable situation of
Proposition~\ref{prop:Moser-integrable}, the canonical relations \eqref{eqn:Moser-canonical-constraints} imply that the integrals $F_{r+1},\dots,F_n$ commute with the constraints $\phi=(G,F)$ in the appropriate sense, so that their Dirac brackets on $N$ coincide with their restrictions of the ambient bracket. Therefore, the integrability of the constrained system in Proposition~\ref{prop:Moser-integrable} can be restated as saying that the Dirac bracket on $M=\mathds R^{2n}$ preserves the commutative algebra generated by the integrals $F_{r+1},\dots,F_n$ along $N$.

\subsection{Near-integrability on a symplectic stratum}
\label{subsec:stratum-near-integrability}

We now explain how Theorem~\ref{thm:general_intertwining_JS} answers Question~\ref{conj: near integrable} on an arbitrary region of a fixed symplectic stratum.

Let $H_\varepsilon\in C^\infty(M)^G$ be a one-parameter family of $G$-invariant Hamiltonians $H_\varepsilon=H_0+\varepsilon H_1$, where $\varepsilon\in(-\varepsilon_0,\varepsilon_0)$ for some $\varepsilon_0>0$. Fix $\mu\in\mathfrak g^*$ such that $J_\mu:=J^{-1}(\mu)$ is a smooth submanifold, and set $M_\mu:=J_\mu/G_\mu$ with quotient map $\pi_{\mu}:J_\mu\to M_\mu$ as in~\eqref{eqn: notation J, Gmu, Mmu} and~\eqref{eqn: quotient map}. Let $\mathcal S\subset M_\mu$ be a symplectic stratum and fix $x_0\in J_\mu$ with $[x_0]:=\pi_{\mu}(x_0)\in\mathcal S$. 

Assume that there exist $\varepsilon_0>0$ and a neighborhood $U\subset J_\mu$ such that, for all $|\varepsilon|<\varepsilon_0$, $H_\varepsilon$ is drift-free near $x_0\in U$ (see Definition~\ref{def:drift-free}). Let $N\subset U$ be the symplectic submanifold defined in~\eqref{eqn: N}, and denote by $\{\cdot,\cdot\}_N$ the corresponding Poisson bracket on $C^\infty(N)$. Let $\iota:N\hookrightarrow J_\mu$ be the inclusion. After shrinking $N$ if necessary, the map
\begin{equation}\label{eqn: projection pi}
\pi:=\pi_{\mu}\circ\iota: N\to U_{\mathcal S}\subset \mathcal S
\end{equation}
is a diffeomorphism onto an open neighborhood $U_{\mathcal S}$ of $[x_0]$ in $\mathcal S$. Then Theorem~\ref{thm:general_intertwining_JS} implies that the Hamiltonian vector fields $X_{H_\varepsilon|N}$ on $N$ and $X_{\widetilde H_\varepsilon}$ on $\mathcal S$ are $\pi$-related: for all $f\in C^\infty(U_{\mathcal S})$,
\begin{equation}\label{eqn:intertwine-N-S}
\{H_\varepsilon|_N,\, f\circ \pi\}_N = \{\widetilde H_\varepsilon,\, f\}_{\mathcal S}\circ \pi,
\end{equation}
where $\widetilde H_\varepsilon$ denotes the reduced Hamiltonian on $\mathcal S$.

\begin{proposition}\label{prop: near-integrable-transfer}
Assume \eqref{eqn:intertwine-N-S} holds on $N$. Then $H_\varepsilon|_N$ is near-integrable on $N$ (see Definition~\ref{def: near-integrable}) if and only if $\widetilde H_\varepsilon$ is near-integrable on $U_{\mathcal S}$. In that case, action-angle coordinates for $\widetilde H_0$ on $U_{\mathcal S}$ pull back via $\pi$ to action-angle coordinates for $H_0|_N$ on $N$, and the perturbation terms are related by composition with the mapping $\pi$ as in~\eqref{eqn: projection pi}.
\end{proposition}

\begin{proof}
Assume that $\widetilde H_\varepsilon$ is near-integrable on $U_{\mathcal S}$ in the sense of Definition~\ref{def: near-integrable}. Then, after possibly shrinking $U_{\mathcal S}$, there exist action-angle coordinates
\[(\theta,I): U_{\mathcal S}\to \mathds T^n\times D\]
such that
\begin{equation}\label{eqn: form-on-V}
  \widetilde H_\varepsilon(\theta,I)=\widetilde H_0(I)+\varepsilon \widetilde H_1(\theta,I;\varepsilon).
\end{equation}
Denote by
\[
  (\Theta,\mathcal I) := (\theta,I)\circ \pi : N\to \mathds T^n\times D.
\]
Since $\pi$ in \eqref{eqn: projection pi} is a diffeomorphism, $(\Theta,\mathcal I)$ is a local coordinate system on $N$. Moreover, because $\widetilde H_\varepsilon\circ \pi = H_\varepsilon|_N$ by construction of the reduced Hamiltonian, \eqref{eqn: form-on-V} implies
\begin{equation}\label{eqn: form-on-N}
  H_\varepsilon|_N(\Theta,\mathcal I)=H_0(\mathcal I)+\varepsilon H_1(\Theta,\mathcal I;\varepsilon),
\end{equation}
so the near-integrable ``functional form'' is preserved under pullback.

It remains to check that the pulled-back coordinates describe the Hamiltonian dynamics of $H_\varepsilon|_N$. Identity \eqref{eqn:intertwine-N-S} implies precisely that, for every $g\in C^\infty(U_{\mathcal S})$,
\[
X_{H_\varepsilon|_N}(g\circ\pi) \;=\; (X_{\widetilde H_\varepsilon}g)\circ\pi,
\]
hence $X_{H_\varepsilon|_N}$ is $\pi$-related to $X_{\widetilde H_\varepsilon}$. 

Moreover, by singular symplectic reduction on the stratum $\mathcal S$, the symplectic form $\omega_{\mathcal S}$ on $U_{\mathcal S}$ is characterized by the identity 
\[\pi_{\mu}^*\omega_{\mathcal S} = \iota_{\mathcal S}^*(\iota^*\omega),\]
where $\iota_{\mathcal S}\colon \pi_{\mu}^{-1}(U_{\mathcal S})\hookrightarrow J_\mu$ denotes the inclusion. Restricting further to the symplectic slice $N$ and using~\eqref{eqn: projection pi} we obtain
\[
\pi^*\omega_{\mathcal S}=\omega|_N.
\]
Hence the action-angle coordinates $(\theta,I)$ for $\widetilde H_0$ on $U_{\mathcal S}$ pull back via $\pi$ to action-angle coordinates $(\Theta,\mathcal I)$ on $N$. Using~\eqref{eqn: form-on-N}, this shows that $H_\varepsilon|_N$ is near-integable on $N$ in the sense of Definition~\ref{def: near-integrable}, with action-angle coordinates obtained by pullback via $\pi$.

The converse implication follows by the same argument, pushing forward action-angle coordinates from $N$ to $U_{\mathcal S}$ via $\pi$.
\end{proof}

\begin{remark}
Even though Theorem~\ref{thm:general_intertwining_JS} is formulated on a neighborhood $U\subset J_\mu$, we work on the slice $N$ because $J_\mu$ is in general only presymplectic: there is no natural Poisson bracket on all of $C^\infty(J_\mu)$, but only on the admissible subalgebra $C^\infty_{\mathrm{adm}}(J_\mu)$ defined in~\eqref{eqn: admissible functions}. In contrast, $N$ is a symplectic submanifold, hence carries a nondegenerate Poisson bracket on all of $C^\infty(N)$, so that standard symplectic notions such as action-angle coordinates and near-integrability apply directly.
\end{remark}

Proposition~\ref{prop: near-integrable-transfer} allows one to apply standard perturbation results on the symplectic stratum $S$ and ``interpret'' the conclusions on $N$ via the map $\pi$ in~\eqref{eqn: projection pi}. In particular, any persistence result that applies to near-integrable Hamiltonians on a symplectic manifold can be invoked on $S$ (or on $N$), and the resulting invariant objects correspond under $\pi$.

\begin{remark}\label{rem:KAM-transfer}
Under the hypotheses of Proposition~\ref{prop: near-integrable-transfer}, any KAM theorem applicable to $\widetilde H_\varepsilon$ on $U_{\mathcal S}$ applies equally to $H_\varepsilon|_N$ on $N$, and vice versa. Hence, for $|\varepsilon|$ sufficiently small, the Cantor families of invariant Lagrangian tori given by the KAM theorem correspond through the symplectomorphism $\pi:N\to U_{\mathcal S}$.
\end{remark}

\section*{Declarations}


\begin{itemize}
\item \textbf{Funding}: L.Zhao is supported by DFG ZH 605-4/1 and Research Funds for Central Universities of China. J.Lamas is supported by the Research Funds for Central Universities of China.





\item \textbf{Competing interests}: The authors have no competing interests to declare that are relevant to the content of this article. All authors certify that they have no affiliations with or involvement in any organization or entity with
any financial interest or non-financial interest in the subject matter or materials discussed in this manuscript.
The authors have no financial or proprietary interests in any material discussed in this article. Indirectly, the
authors hope to gain reputation in the scientific community via the publication of this article.
\item \textbf{Data availability} : The authors declare that all the data supporting the findings of this study are available within the paper.
\end{itemize}

\printbibliography
\end{document}